\documentclass[reqno]{amsart}

\usepackage{epsfig, enumerate}
\usepackage{graphicx}
\usepackage[usenames,dvipsnames,svgnames,table]{xcolor}
\usepackage{cases}
\usepackage{mathtools}
\usepackage{a4wide}
\usepackage[charter]{mathdesign}
\usepackage{comment}
\usepackage{todonotes}
\usepackage[displaymath, mathlines]{lineno}
\usepackage{booktabs}
\usepackage{bbm}

\usepackage{cancel}

\usepackage{tikz}
\usetikzlibrary{backgrounds}
\usetikzlibrary{patterns,fadings}
\usetikzlibrary{arrows,decorations.pathmorphing}
\usetikzlibrary{calc}
\definecolor{light-gray}{gray}{0.95}
\usepackage{float}
\usepackage[colorlinks=true,linkcolor=blue,citecolor=magenta]{hyperref}

\newcommand{\lj}[1]{#1}
\newcommand{\mm}[1]{#1}

\newtheorem{theorem}{Theorem}[section]
\newtheorem{lemma}[theorem]{Lemma}

\newtheorem{remark}[theorem]{Remark}
\newtheorem{definition}{Definition}[section]

\numberwithin{equation}{section}

\newcommand{\mc}[1]{{\mathcal #1}}

\newcommand{\bb}[1]{{\mathbb #1}}

\newcommand{\<}{\langle}
\renewcommand{\>}{\rangle}

\renewcommand{\epsilon}{\varepsilon}

\def\centerarc[#1](#2)(#3:#4:#5){\draw[#1] ($(#2)+({#5*cos(#3)},{#5*sin(#3)})$) arc (#3:#4:#5);}

\renewcommand{\geq}{\geqslant}
\renewcommand{\leq}{\leqslant}

\renewcommand\bar{\overline}
\renewcommand{\tilde}{\widetilde}

\newcommand{\R}{\mathbb{R}}
\newcommand{\Q}{\mathbb{Q}}
\newcommand{\Z}{\mathbb{Z}}
\newcommand{\N}{\mathbb{N}}

\newcommand{\E}{\bb E}

\newcommand{\Lsink}{\mc L_N^{\alpha,\beta}}

\newcommand{\Lex}{\mc L_N^{\text{\rm ZR}}}

\let\oldtocsection=\tocsection
\let\oldtocsubsection=\tocsubsection
\let\oldtocsubsubsection=\tocsubsubsection
\renewcommand{\tocsection}[2]{\hspace{0em}\oldtocsection{#1}{#2}}
\renewcommand{\tocsubsection}[2]{\hspace{1em}\oldtocsubsection{#1}{#2}}
\renewcommand{\tocsubsubsection}[2]{\hspace{2em}\oldtocsubsubsection{#1}{#2}}
\DeclareRobustCommand{\SkipTocEntry}[5]{}

\newcommand{\dd}{{\mathfrak{d}}}
\newcommand{\pa}[1]{\left(#1 \right)}

\newcommand{\cro}[1]{\left[#1\right]}
\newcommand{\br}[1]{\left\{#1\right\}}
\newcommand{\Prob}{{\mathbb{P}}}

\newcommand{\red}[1]{{\color{black} #1}}

\newcommand{\ccl}[1]{{\color{black} #1}}

\begin{document}

\allowdisplaybreaks

\title{Asymmetric attractive zero-range processes with particle destruction at the origin}


\author{Cl\'ement Erignoux}
\address{Inria, Univ. Lille, CNRS, UMR 8524 - Laboratoire Paul Painlev\'e, F-59000 Lille}
\email{clement.erignoux@inria.fr}
\author{Marielle Simon}
\address{Université Claude Bernard Lyon 1, CNRS UMR 5208, Institut Camille Jordan, F-69622 Villeurbanne, France
\textit{and} GSSI, L'Aquila, Italy}
\email{msimon@math.univ-lyon1.fr}
\author{Linjie Zhao}
\address{School of Mathematics and Statistics, Huazhong University of Science \& Technology, Luoyu Road 1037, Wuhan 430074, China.}
\email{linjie\_zhao@hust.edu.cn}

\thanks{
\textsc{Acknowledgments}: The authors would like to thank Christophe Bahadoran for his help navigating asymmetric zero-range processes. This project is partially supported by the ANR grant MICMOV (ANR-19-CE40-0012)
of the French National Research Agency (ANR), and by the European Union with the program Fonds européen de développement régional. It has also received funding from the European
Research Council (ERC) under the European Union’s Horizon 2020 research and innovative program
(grant agreement n o 715734),  from
Labex CEMPI (ANR-11-LABX-0007-01), and from the Fundamental Research Funds for the Central Universities in China.}


\begin{abstract} We investigate the macroscopic behavior of asymmetric attractive zero-range processes on $\Z$ where particles are destroyed at the origin at a rate of order $N^\beta$, where $\beta \in \R$  and $N\in\N$ is the scaling parameter. We prove that the hydrodynamic limit of this particle system is described by the unique entropy solution of a hyperbolic conservation law, supplemented by a boundary condition depending on the range of $\beta$. Namely, if $\beta \geqslant 0$, then the boundary condition prescribes the particle current through the origin, whereas if $\beta<0$, the destruction of particles at the origin has no macroscopic effect on the system and no boundary condition is imposed at the hydrodynamic limit.

\end{abstract}

\maketitle


\section{Introduction}

Asymmetric and conservative particle systems are among the most interesting microscopic dynamics  currently  under investigation in statistical physics. Their macroscopic behavior, due to the microscopic asymmetry of the particles' motion, is generally expected to be described by hyperbolic equations, whose solutions may be visualized as propagating waves. When the system is nonlinear, jump discontinuities, known as \emph{shocks}, can arise in finite time, even if the initial profile is smooth. Even for linear systems, discontinuities can appear in presence of boundary effects. Therefore, when asymmetric particle systems are subject to strong boundary mechanisms, their macroscopic behavior is much less understood. In a general setting, imposing given boundary conditions at the macroscopic level starting from a given microscopic model is far from trivial.  Because of discontinuities, hyperbolic equations do not in general have a unique solution; in order to fully describe the macroscopic limit of a given microscopic asymmetric system, one needs to rule out the non-physical solutions. The unique relevant one, called \emph{entropy solution}, can be characterized in several different ways, and exhibits physically consistent behavior after the shocks (see the reference book \cite{malek1996}, in particular Chapters 2.6 and 2.7 for the existence and uniqueness of solutions to hyperbolic conservation laws in bounded domains). \mm{The name \emph{entropy} comes from the additional criterion, \textit{aka} the \emph{entropy inequality}, which is imposed in order to ensure the uniqueness of a weak solution, and, from the physical viewpoint, select the `correct' physical
solution among all weak ones \ccl{(see Remark \ref{rem:heur} below, and \cite[Chapter 2.3]{malek1996} for a heuristic derivation of this inequality).}}

The characterization of the macroscopic behavior of asymmetric particle systems  by hyperbolic equations has been initiated by Rezakhanlou in \cite{rezakhanlou91}. In the latter, the author derives the \emph{hydrodynamic limit} -- the law of large \lj{numbers} characterizing a particle system's macroscopic behavior in terms of a given PDE -- in cases where the microscopic dynamics are \emph{attractive}. He proved in particular that  the macroscopic density $\rho(t,u)$ of the Asymmetric Simple Exclusion Process (ASEP), in the hyperbolic time scale is given by the unique entropy solution to the hydrodynamic equation $\partial_t\rho=\partial_u(\rho(1-\rho))$ on the full line $\R$. Twenty years later Bahadoran \cite{bahadoran2012hydrodynamics} considered  a class of asymmetric lattice gases, including the ASEP, with boundary interactions and obtained in a fairly general setting the convergence to
the unique entropy solution on a bounded domain, satisfying the \emph{Bardos-Leroux-N\'ed\'elec boundary condition} \cite{bardos1979} in the sense of Otto \cite{otto96}. Very recently, De Masi \emph{et al.}~\cite{demasi2021}  have carried on with the study of the open ASEP, and considered its \emph{quasi-static} hydrodynamic limit with time dependent boundary rates changing on a slow  time scale (w.r.t.~the hyperbolic one). The quasi-static limit they obtain characterizes the slow transition between non-equilibrium stationary states due to the influence of the boundary. They prove in \cite{demasi2021} that the density profile
converges to the entropy solution of the quasi-static Burgers equation with the corresponding time dependent boundary condition.  Finally, Xu \cite{xu2021hydrodynamic} has further  derived the hydrodynamic limit for the open ASEP with accelerated boundaries perturbed by stirring dynamics.

In the case of asymmetric \emph{zero-range} processes,  the law of large numbers for the particle density on the full line has also been obtained in the attractive case by Rezankhanlou in \cite{rezakhanlou91}. Then, Landim \cite{landim1996hydrodynamical} focused on  the \emph{totally
 asymmetric} case and introduced a boundary effect at the origin by slowing down the jump rate at 0. He  proved the convergence of the density profile towards the unique entropy solution of the hyperbolic conservation law with mass creation/annihilation at the origin. Bahadoran's later work \cite{bahadoran2012hydrodynamics} also encompassed general zero-range processes, and therefore, from his work follows the hydrodynamic limit for the asymmetric zero-range process on the real line and in presence of boundary reservoirs.

In the present article, we build on Bahadoran's work \cite{bahadoran2012hydrodynamics} by considering  general  asymmetric zero-range processes on $\Z$ and we introduce a microscopic mechanism which destroys particles at the origin: let  $g:\N\to\R_+$ be the jump rate function, $\alpha \geqslant 0$ and $\beta \in\R$ be two parameters that adjust the intensity of the destruction dynamics, and $N\in\N$ be the hydrodynamic limit's scaling parameter. Then, the particle system evolves as follows:\begin{itemize} \item (\textit{bulk dynamics}) if there are $\omega_x \in \N$ particles at site $x\in\Z$, then one of these  particles jumps to its right (resp.~left) neighbour at rate $pg(\omega_x)$  (resp.~$(1-p)g(\omega_x)$), \smallskip
\item (\textit{boundary effect})  at the origin  $x=0$, one particle is destroyed at rate $\alpha N^\beta g(\omega_0)$.\end{itemize}   
The purpose of this paper is twofold: first, we  explore the effect of the ``boundary'' dynamics at the origin on the asymmetric zero-range processes on the full line. Because the microscopic destruction mechanism is \emph{inside} the domain, and not at its border as in \cite{bahadoran2012hydrodynamics}, there is a non-trivial interplay between the two resulting half-lines. Note that we consider a  destruction dynamics whose rate $N^\beta$ varies (w.r.t.~the asymmetric jump dynamics), in the spirit of the slow boundaries recently thoroughly investigated for symmetric exclusion processes \cite{BMN17, FGN15}. Our main result is the derivation of a family of hyperbolic equations on the one-dimensional space $\R$, depending on the relative strength of the destruction dynamics. As is the case for the Simple Symmetric Exclusion Process (SSEP)  (see \emph{e.g.}~\cite{BMN17}), we show that the macroscopic behavior of the asymmetric zero-range process at the boundary depends on the value of $\beta$ (see  Theorem \ref{thmAttrac} for a precise statement): \begin{enumerate}[(i)] \item if $\beta <0$ (subcritical case),  the particle destruction has no macroscopic effect and the hydrodynamic equation is given by the hyperbolic conservation law on the full line $\R$ with no boundary condition; 
\item if $\beta= 0$ (critical case),  the hyperbolic conservation law is now supplemented by a boundary condition which prescribes the particle current allowed through the origin (as an explicit function of the density left of the origin), roughly analogous to the Robin boundary condition obtained for the SSEP \cite{BMN17}; \item if $\beta>0$ (supercritical case), this current through the origin vanishes at the macroscopic scale  and the system then enforces some Dirichlet type  boundary condition, whose effect is to remove all mass from a growing segment $[0,(2p-1)t]$. \end{enumerate}

The second purpose of this article is to provide a simple proof for the hydrodynamic limit in the linear case where $g(k)=k$. The solution to the hydrodynamic equation then becomes an explicit function of the initial density profile in the system. In this particular case, we provide explicit formulas and a short and new proof for the hydrodynamic limit based on duality arguments inspired by previous work \cite{Erignoux18, ELX18} (see Section \ref{app}).

All the results mentioned above \cite{bahadoran2012hydrodynamics,demasi2021,landim1996hydrodynamical,rezakhanlou91} strongly rely on the \emph{attractiveness property}  of both the ASEP and zero-range processes. Here, we keep attractiveness as a crucial assumption.  Indeed, attractive zero-range processes have nice  features which allow us to introduce the notion of \emph{second-class particles} and to use coupling techniques. Similarly to \cite{demasi2021} we need to control the microscopic boundary entropy flux and we use the notion of  boundary entropy-entropy flux
pairs introduced by Otto \cite{otto96} and used in \cite{landim1996hydrodynamical}  in order to characterize the boundary conditions in the scalar hyperbolic equations.
 For non-attractive asymmetric processes, very little is known: in fact,  Yau's classical \emph{relative entropy} approach \cite{yau1991relative} is the unique method that derives the hydrodynamic behavior of such purely asymmetric processes in general, but it requires the solution to the hydrodynamic equation to be smooth. For this reason, it can
 usually only be applied to asymmetric systems up to the time of the first shock, and furthermore fails completely in the presence of boundary conditions such as those considered in this work. Another option, which would allow to derive non-attractive hydrodynamic limits for asymmetric systems in the presence of discontinuities (past the first shock, in particular) would be to adapt Fritz's \emph{compensated compactness arguments} \cite{fritz2004,fritztoth}. This technique, however, is not without its own shortcomings, and in particular requires to perturb the underlying microscopic dynamics by some symmetric stirring dynamics. This perturbation needs to be strong enough to ensure that one recovers at the hydrodynamic limit the entropy (or vanishing viscosity) solution to the hyperbolic equation, but weak enough so that the symmetric part of the dynamics does not appear in the limiting equation. This work is currently in progress in order to drop the attractiveness assumption in our model.
 
The article is organized as follows: in Section \ref{sec:model} we introduce the microscopic model and state the main result of the paper, namely the hydrodynamic limit for our attractive zero-range process with destruction of particles at the origin. In Section  \ref{sec:prelim} we prove preliminary results on second-class particles and present the coupling argument, which will be used in Section \ref{sec:proof} to prove the main theorem. In Section \ref{app} we present a short alternative proof in the specific case where the jump rate is linear, using duality arguments, and in Appendix \ref{app:equiv} we prove a technical result on the equivalence of two notions of entropy solutions.

\section{Definitions, notations, and results}\label{sec:model}

\subsection{Zero-range process with source dynamics} 
Let $N$ be a positive integer and $\Omega = {\bb N}^{\bb Z}$  be the state space of the Markov process  $(\omega(t))_{t\geq 0}$ which will be the focus of this article. \lj{Its} dynamics is characterized below by its infinitesimal generator $\mathcal{L}_N$. Given a  configuration of particles $\omega \in \Omega$, $\omega_x \in \N$ represents the number of particles at site $x\in\Z$.  Before defining the generator $\mc L_N$, let us fix: 
\begin{itemize}\item  two parameters  $\alpha\geq 0$ and $\beta \in\R$, which adjust  the intensity of the particle destruction at the origin;
\item the asymmetry parameter $p \neq \frac12$. Without loss of generality, we assume $\frac12 < p \leq 1$;
\item a function $g: \mathbb{N} \rightarrow \mathbb{R}_+$ which represents the rate at which one particle leaves a site. This rate function satisfies $g (0) = 0$, $g (k) > 0$ for any $k \geq 1$, and is assumed to be \emph{Lipschitz}: there is a constant $a_0>0$ such that
\begin{equation}
\label{eq:lipg}
\sup_{k \geq 0} |g (k+1) - g(k)| \leq a_0.
\end{equation} 
We consider in the present work the \emph{attractive case}, which means that we further assume that
\begin{equation}
\label{ass:H1} \tag{H1}
\mbox{the function $g$ is non-decreasing, \textit{i.e.}~$g (k) \leq g (k+1)$ for all $k \geq 0$.} 
\end{equation}
 \end{itemize}
We say that a function $f:\Omega \to \bb R$ is \emph{local} if it depends on $\omega=\{\omega_x\}_{x\in\bb Z} \in \Omega$ only through a finite number of coordinates $(\omega_{x_1},\dots,\omega_{x_k})$, and a local function $f:\Omega \to \bb R$ is  \emph{Lipschitz} if there is a constant $c>0$ and a finite subset $\Lambda\subset \Z$ such that
\[
|f(\eta) - f(\zeta)| \leq c \sum_{x\in\Lambda}|\eta(x) - \zeta(x)|, \quad \text{ for all } \eta,\,\zeta\in \Omega.
\]
The generator $\mc L_N$ of the zero-range process considered here acts on local Lipschitz functions $f: \Omega \to \bb R$ by
\begin{equation}\label{eq2.1}
{\mathcal L}_N f(\omega) =  \Lex f(\omega)+\Lsink f(\omega).
\end{equation}
The generator $\Lex$ of the \emph{asymmetric zero-range dynamics} is given by
\begin{equation}\label{lex}
\Lex f(\omega) =  \sum_{x \in \bb Z}g(\omega_x)  \Big\{ p \nabla_{x,x+1} f(\omega) + (1-p) \nabla_{x,x-1} f(\omega) \Big\}
\end{equation}
where for $\omega\in\Omega$ and $x,y\in\mathbb Z$, we denote $\nabla_{x,y} f(\omega)=f(\omega^{x,y})-f(\omega)$ and $\omega^{x,y} \in \Omega$ is the configuration obtained from $\omega$ after a particle jumps from site $x$ to site $y$,
\[
\omega^{x,y}_z=
\left\{
\begin{array}{rl}
\omega_x-1& \text{ if } z=x,\\
\omega_y+1& \text{ if } z=y,\\
\omega_z& \text{ if } z\neq x,y.\\
\end{array}
\right.
\]
The generator $\Lsink$ of the \emph{source dynamics} is given by
\[
\Lsink f(\omega)\;=\;\alpha N^\beta g(\omega_0)\; \Big\{ f(\omega^{(0)})- f(\omega)\Big\}\,,
\]
where $\omega^{(0)} \in \Omega$ is the configuration obtained from $\Omega$ after the destruction of one particle at the origin
\[
\omega^{(0)}_x=\omega_x-{\bb 1}\{x=0\},
\]
$\beta\in \R$, and $\alpha>0$ is a positive parameter, fixed throughout. We refer the readers to \cite{Andjel82} for a rigorous construction of this infinite-volume dynamics, when $\alpha=0$. The general case $\alpha \neq 0$ is a straightforward adaptation of \cite{Andjel82}. 

\medskip

\subsection{Invariant measures.}\label{sebsec_invariant} We now describe the invariant measures for the zero-range process \emph{without} particle destruction. To that aim, following \cite[Chapter 2]{kipnis1989hydrodynamics}, let us introduce the partition function $Z:\R_+ \to \R_+$ given by
\[
Z (\zeta) = \sum_{k \geq 0} \frac{\zeta^k}{g(k)!},  \qquad \text{for } 0 \leq  \zeta < \zeta^*,
\]
where $\zeta^*$ is the radius of convergence of the above summation and $g(k)!:=\prod_{\ell=1}^k g(\ell)$ with the convention $g(0)!=1$.   For each $0 \leq  \zeta < \zeta^*$, let $\bar{\nu}_\zeta$ be the product measure on $\Omega$ with marginals given by 
\[
\bar{\nu}_{\zeta}(\omega_x = k) = \frac{1}{Z (\zeta)}\; \frac{\zeta^k}{g(k)!}, \quad x\in\Z, \,k \in\N.
\]
It is easy to check that $\bar{\nu}_{\zeta}$, $0 \leq \zeta < \zeta^*$, are invariant measures for the generator $\Lex$. Furthermore, the particle density $R (\zeta) = \int \omega_0 \bar{\nu}_{\zeta} (d \omega)$ is strictly increasing in $\zeta$, hence has an inverse denoted by $\Phi(\cdot)$. To index by density the invariant measures of the generator $\Lex$, we denote $\nu_\rho = \bar{\nu}_{\Phi (\rho)}$ and note that \begin{equation}\Phi (\rho) = \int g (\omega_0)\,\nu_\rho (d \omega).\label{eq:phi}\end{equation}
Under our assumptions, $\Phi: \R_+ \to\R_+$ is smooth and non-decreasing, where $\R_+=[0,+\infty)$. Throughout, when $0$ needs to be excluded from $\R_+$, we will write $(0,+\infty)$.

The invariant measures for the full process generated by $\mathcal{L}_N$ will be given in Section \ref{sec:invariantm}. In the following, for a probability measure $\nu$ on a state space $\Omega$, we denote by $E_{\nu}[f]$ the expectation of a function $f$ defined on $\Omega$ with respect to $\nu$.

\subsection{Entropy solutions}

%
%
%

\medskip

To state our main result, we first need to introduce the notion of \emph{entropy solutions} to the hydrodynamic limit in the various cases for the parameter $\beta$. We start with the standard definition when no boundary condition is imposed, and then we define the entropy solution  to  \emph{initial-boundary value problems}, as given for instance in \cite{bahadoran2012hydrodynamics,landim1996hydrodynamical}. In the following we denote by $C_K^{k,\ell}(U\times V)$ the set of functions $H:U\times V\to\R$ which are compactly supported, and $C^k$-regular (resp.~$C^\ell$) in the first (resp.~second) variable.

\begin{definition}
[Hydrodynamic equation  on $\R$ for  $\beta < 0$] 
\label{def:ES1}
Assume that $\rho \mapsto \Phi(\rho)$ is a smooth function.
	We say that $\rho \in L^\infty (\R_+\times\R)$ is an entropy solution to 
\begin{equation}\label{hEneg}
	\begin{cases}
		\partial_t \rho (t,u) + (2p-1) \partial_u \Phi (\rho (t,u)) = 0,&\quad t > 0, u \in \bb{R},\\
		\rho (0,u) = \rho_0 (u),&\quad u \in \mathbb{R},
	\end{cases}
\end{equation}
if 
\begin{itemize}
\item for any non-negative test function $H \in C_K^{1,1}((0,+\infty)\times\R)$, and any constant $c\geqslant 0$,
	\begin{equation}\label{hEnegInteg}
		\begin{aligned}
			\int_0^\infty \int_{\R}  \left\{\partial_t H (t,u) \big|\rho(t,u)-c\big|+  (2p-1)\partial_u H (t,u) \big|\Phi (\rho(t,u)) - \Phi(c)\big|\right\} \,du\,dt \geq 0,
		\end{aligned}
	\end{equation}

\item for every $A > 0$,
\[
	\lim_{t \rightarrow 0} \int_{-A}^A\, \big|\rho (t,u) - \rho_0 (u)\big| \,du = 0.
\]
\end{itemize}
\end{definition}

\begin{remark}\label{rem:regularity} It is known that 
entropy solutions are continuous in $(t,u)$ for all but  at most a countable number of shock lines $(t,u)$ (see \cite[Chapter 2]{malek1996} for example). Therefore there exists a version $\rho$  solution to \eqref{hEneg} which is continuous at $(t,0)$ for
all but a countable number of $t$. Then it is easy to check
\[\int_{\R_+} \big\{\rho(t,u) - \rho_0 (u)\big\}\,du  = \int_0^t (2p-1) \Phi (\rho (s,0))\,ds\]
so that in particular 
\[\partial_t \int_{\R_+} \big\{\rho(t,u) - \rho_0 (u)\big\}\,du =  (2p-1) \Phi (\rho (t,0))\]
and $t\mapsto \Phi(\rho(t,0))$ is almost everywhere continuous on $\R_+$.
\end{remark}

\begin{remark} \label{rem:heur}
\mm{Let us recall here the standard heuristics for the derivation of  \eqref{hEnegInteg}.}
For any $\varepsilon > 0$,	consider the following \ccl{viscous} approximation of the equation \eqref{hEneg},
	\[
		\begin{cases}
			\partial_t \rho^\varepsilon (t,u) + (2p-1) \partial_u \Phi (\rho^\varepsilon (t,u)) = \varepsilon \partial_u^2 \rho^\varepsilon (t,u),&\quad t > 0, u \in \bb{R},\\
			\rho^\varepsilon (0,u) = \rho_0 (u),&\quad u \in \mathbb{R}.
		\end{cases}
	\]
Let $E: \R \rightarrow \R$ be continuously differentiable and convex. Let  $F: \R \rightarrow \R$ be such that $(2p-1) \Phi^\prime E^\prime = F^\prime$.  The pair $(E,F)$ is called an entropy-entropy flux pair in the literature. One could check directly that
\[\partial_t E(\rho^\varepsilon) + \partial_u F(\rho^\varepsilon) = \varepsilon \partial_u^2 E(\rho^\varepsilon) - \varepsilon E^{\prime \prime} (\rho^\varepsilon) [\partial_u \rho^\varepsilon]^2.\]  
Since $E$ is convex, 
\[\partial_t E(\rho^\varepsilon) + \partial_u F(\rho^\varepsilon) \leq \varepsilon \partial_u^2 E(\rho^\varepsilon).\]  
\ccl{Letting $\varepsilon \rightarrow 0$, called in the litterature the \emph{vanishing viscosity limit,} and assuming $\rho^\varepsilon$ converges  to  the solution $\rho$} of \eqref{hEneg},
\[\partial_t E(\rho) + \partial_u F(\rho) \leq 0\]  
in the sense of distributions. \ccl{Although this bound must hold for any entropy-entropy flux pair $(E,F)$ for the physically relevant solution, in order to select the latter, it is enough that it holds  for all pairs $E(\rho) = |\rho-c|$ and $F(\rho) = (2p-1) |\Phi(\rho) - \Phi (c)|$, where $c > 0$, which is precisely the entropy inequality \eqref{hEnegInteg}}. 
\end{remark}

In the cases where $\beta\geq 0$, we will need to separate both parts of the real line because of the boundary condition at the origin.
\begin{definition}
[Hydrodynamic equation on $\R^+$ for $\beta = 0$] 
\label{def:ES2}
Assume that $\rho \mapsto \Phi(\rho)$ is a smooth  {non-decreasing} function. Let $\varrho: \R_+ \rightarrow \R_+$ be an almost-everywhere continuous function. We say that $\rho \in L^\infty (\R_+ \times \mathbb{R}_+)$  is an entropy solution to  

	\begin{equation}\label{hE0}
		\begin{cases}
			\partial_t \rho (t,u) + (2p-1) \partial_u \Phi (\rho (t,u)) = 0,&\quad t > 0, u > 0,\\
		\rho (t,0) = \varrho (t),&\quad t > 0,\\
			\rho (0,u) = \rho_0 (u),&\quad u \in \mathbb{R}_+,
		\end{cases}
	\end{equation}
	if  
	\begin{itemize}
		\item there exists $M > 0$ such that for any non-negative test function $H \in C_K^{1,1} ((0,\infty) \times \R)$ and any constant $c \geq 0$,
		
		\begin{equation}\label{hE0Integ1}
			\begin{aligned}
						&\int_0^\infty \int_0^\infty  \big\{ \partial_t H (t,u) (\rho(t,u)-c )^{\pm}+ (2p-1)\partial_u H(t,u) (\Phi (\rho(t,u)) - \Phi(c) )^{\pm}\big\}  \,du\,dt \\
						&\qquad + M \int_0^\infty H (t,0) (\varrho (t)-c)^{\pm} \,dt \geq 0,
			\end{aligned}
		\end{equation}
	where for $r \in \bb{R}$, $r^+ = \max \{r,0\}$ and $r^- = \max \{-r,0\}$;
			\item for every $A > 0$ 
			
		\begin{align*}
			\lim_{t \rightarrow 0} \int_0^A \big|\rho (t,u) - \rho_0 (u)\big| \,du = 0.
		\end{align*}
	\end{itemize}
\end{definition}

\begin{definition}[Hydrodynamic equation  on $\R^+$ for $\beta > 0$]\label{def:sol}
Assume that $\rho \mapsto \Phi(\rho)$ is a smooth  non-decreasing function. Let $f: \R_+ \rightarrow \R_+$ be an almost-everywhere continuous function. We say that $\rho \in L^\infty (\R_+ \times \mathbb{R}_+)$  is an entropy solution to  

\begin{equation}\label{hEpos}
	\begin{cases}
		\partial_t \rho (t,u) + (2p-1) \partial_u \Phi (\rho (t,u)) = 0,&\quad t > 0, u > 0,\\
			\partial_t \int_{\R_+} \rho (t,u) du = f(t),&\quad t > 0,\\
		\rho (0,u) = \rho_0 (u),&\quad u \in \mathbb{R}_+,
	\end{cases}
\end{equation}
if  
\begin{itemize}
	\item  for any non-negative test function $H \in C_K^{1,1} ((0,\infty) \times (0,\infty))$ and any constant $c \geq 0$,
	
	\begin{equation}\label{hEposInteg1}
\int_0^\infty \int_0^\infty  \big\{ \partial_t H (t,u) \big|\rho(t,u)-c\big| + (2p-1)\partial_u H(t,u) \big|\Phi (\rho(t,u)) - \Phi(c)\big|\big\}  \,du\,dt \geq 0;
	\end{equation}

   \item for every $A > 0$ and every $T > 0$,
   
   \begin{align}
&   	\lim_{t \rightarrow 0} \int_0^A \big|\rho (t,u) - \rho_0 (u)\big| \,du = 0, 
   \\ \label{hEposBoundary}
&\lim_{u \rightarrow 0} \int_0^{T} \big|(2p-1) \Phi (\rho (t,u) )- f (t)\big| \,dt = 0.
\end{align}
\end{itemize}
\end{definition}

\begin{remark}[Uniqueness of the entropy solution] By Kru{\v z}kov's uniqueness Theorem \cite[Theorem A.2.5.3]{klscaling}, the entropy solution to \eqref{hEneg} is unique. We refer the readers to \cite[Theorem 2.1]{bahadoran2012hydrodynamics}  for the uniqueness of the entropy solution to \eqref{hE0},  and to \cite[Theorem A.1]{landim1996hydrodynamical}\,  for \eqref{hEpos}.
\end{remark}

We adress in Appendix \ref{app:equiv} the equivalence of  Definitions \ref{def:ES2} and \ref{def:sol}. We are now ready to state the main result of this paper.

\subsection{Statements of main results.}
\label{subsection:mainresult}

Fix an initial density profile $\rho_0$ on $\R$, we assume that $\rho_0$ is bounded and Riemann-integrable on any finite segment of $\R$, and we define 
\[\mu_N(d\omega)=\bigotimes_{x\in\Z}\mathcal{P}_{\rho_0(x/N)}(d\omega_x)\]
the product measure with marginal at site $x$ given by the Poisson distribution $\mathcal{P}_{\rho_0(x/N)}$ on $\N$, with parameter $\rho_0(x/N)$.

\medskip

We consider the process $(\omega(t))_{t\geq 0}$ in the hyperbolic scaling, \textit{i.e.}~driven by the \emph{accelerated} generator $N\mathcal{L}_N$, and with initial distribution given by $\mu_N$ on $\Omega$. Denote by $\Prob_{\mu_N}$ the distribution of this process $(\omega(t))_{t\geq 0}$ on $\Omega$, and $\E_{\mu_N}$ the corresponding expectation. We are now ready to state our main result.

\begin{theorem}[Hydrodynamic limit]\label{thmAttrac}
For any compactly supported and  continuous function $H:\R\to\R$, for any $t \geq 0$ and for any $\delta > 0$,
\[
	\lim_{N \rightarrow \infty}  \Prob_{\mu_N}  \pa{\bigg| \frac{1}{N}\sum_{x \in \Z} H \left(\tfrac{x}{N}\right) \omega_x(t) - \int_{\R} H(u) \rho (t,u)\,du \bigg| > \delta }= 0,
\]
where
\begin{enumerate}[(i)]
	\item if $\beta < 0$, then $\rho (t,u)$ is the unique entropy solution to \eqref{hEneg};
	\item if $\beta = 0$, then   $\rho (t,u) = \rho_L (t,u) \mathbb{1} \{u < 0\} +  \rho_R (t,u) \mathbb{1} \{u \geq 0\}$, where $\rho_L$ is the unique entropy solution to \eqref{hEneg} and $\rho_R$ is the unique entropy solution to \eqref{hE0} with 
	
	\begin{equation}
	\label{eq:f} 
	\varrho(t) = R \left(\frac{2p-1}{2p-1+\alpha}\; \Phi \big( \rho_L (t,0)\big)\right),
	\end{equation}
	where the function $R$ is the inverse of $\Phi$ (recall \eqref{eq:phi});
	\item  if $\beta > 0$, then  $\rho (t,u) = \rho_L (t,u) \mathbb{1} \{u < 0\} +  \rho_R (t,u) \mathbb{1} \{u \geq 0\}$, where $\rho_L$ is the unique entropy solution to \eqref{hEneg} and $\rho_R$ is the unique entropy solution to \eqref{hEpos} with $f(t) = 0$.
\end{enumerate}
\end{theorem} 

\begin{remark}
Note that from Remark \ref{rem:regularity}, $\varrho$ defined in \eqref{eq:f} is almost everywhere continuous, so that $\rho_R$ is well defined, as given in Definition \ref{def:ES2}.
\end{remark}

\begin{remark}[Linear case]\label{rem:lincase} 
If we assume $g (k) = k$ for all $k \geq 0$, namely the \emph{linear case}, then the solution can be explicitly computed. To that aim, let us introduce $\widetilde{\alpha}=\widetilde{\alpha}(\beta)$ as

	\begin{equation}
		\label{eq:Defalphat}
		\widetilde{\alpha}=
		\begin{cases}
			1& \mbox{ if } \beta>0,\\
			\alpha/(\alpha+2p-1)& \mbox{ if } \beta=0,\\
			0& \mbox{ if } \beta<0.
		\end{cases}
	\end{equation}
	Then, $\rho(t,u)$ is a weak solution to the formal equation 
	\[
		\begin{cases}
			\partial_t \rho (t,u) + (2p-1) \partial_u \rho (t,u) = -\widetilde{\alpha} \rho(t,0)\delta_0(u),&\quad t > 0, u\in \R,\\
			\rho (0,\cdot) = \rho_0, & \quad u \in \R,
		\end{cases}
\]
and is given, for any $(t,u) \in \R_+ \times \R$,  by: 

	\begin{equation}
		\label{eq:Defrho}
		\rho (t,u) = \big(1 - \widetilde{\alpha}\; \mathbb{1}\{0\leqslant u <(2p-1)t\}\big)\; \rho_0 (u - (2p-1)t).
	\end{equation}	
We provide in Section \ref{app} an alternative (short) proof of Theorem \ref{thmAttrac} in the linear case, using duality tools.
\end{remark}

The rest of the paper is devoted to the proof of Theorem \ref{thmAttrac}.

\section{Preliminary results} \label{sec:prelim}

In this section we give some preliminary results  which will be crucial to prove Theorem \ref{thmAttrac}. We start in Section \ref{sec:invariantm} by defining the invariant measures for the zero-range process with destruction at the origin. Then, in Section \ref{sec:secondclass} we  introduce the notion of \emph{second-class particles}.  In Section \ref{sebsection:coupling} we prove an entropy inequality at the microscopic level  using  coupling techniques.

\subsection{Invariant measures for the process with destruction}
\label{sec:invariantm} For a function $m: \mathbb{Z} \rightarrow \mathbb{R}_+,$ let $\bar{\nu}_{m (\cdot)}$ be the product probability measure on $\Omega$ with marginals given by
\begin{equation}
	\label{eq:mesinv}
	\bar{\nu}_{m (\cdot)}(\omega_x = k) = \bar{\nu}_{m(x)}(\omega_x = k), \quad  x\in\Z, \,k \in\N.
\end{equation}
In order for $\bar{\nu}_{m (\cdot)}$ to be invariant for the whole process generated by $\mathcal{L}_N$, the function $m$ needs to solve some discrete system, as stated by the following lemma: 

\begin{lemma}\label{lem2}
Assume that $m:\Z\to\R_+$ is solution to 
\begin{equation}\label{a1}
\begin{cases}
p m_{x-1}+ (1-p) m_{x+1} - m_x = 0 &  \mbox{ for } x \neq 0,\\
p m_{-1} + (1-p) m_1 - (1+\alpha N^{\beta}) m_0 = 0.
\end{cases}
\end{equation}
Then the product measure $\bar{\nu}_{m (\cdot)}$  is invariant for the dynamics generated by ${\mathcal L}_N$.
\end{lemma}

\begin{remark}\label{remark:invariantm}
$(i)$ If  $p = 1$, namely the \emph{totally asymmetric case}, the solution to \eqref{a1} is
\[
			m_x =
	\begin{cases}
 m_-,\quad \text{if } x \leq -1,\\
	 m_+,\quad \text{if } x \geq 0,
	\end{cases}
\]
	where $m_-,\,m_+$ are non-negative and satisfy $m_- = (1+\alpha N^{\beta}) m_+$.  \\
$(ii) $ If $\frac12 < p < 1$, then any solution to \eqref{a1} has the following form:
\[
	m_x = 
	\begin{cases}
	\displaystyle c_1\; \tfrac{p^x}{(1-p)^x} + c_2 &\quad \text{if $x \leq 0$}, \vphantom{\bigg(}\\
	\displaystyle c_3\; \tfrac{p^x}{(1-p)^x} + c_4 &\quad \text{if $x \geq 0$},\\
	\end{cases}
\]
	where $\{c_i\}_{1 \leq i \leq 4}$ are  real numbers  such that $m_x \geq 0$ for all $x \in \bb{Z}$, and moreover they satisfy
\[	\begin{cases}
	c_1 + c_2 = c_3 + c_4,\\
	c_1 (1-p) + c_2 p + c_3 p + c_4 (1-p)  = (1 + \alpha N^{\beta}) (c_1 + c_2).
	\end{cases}
\]
Solving the above equations, we have
\[
		\begin{cases}
			c_3 = c_1 + \frac{\alpha N^{\beta}}{(2p-1)} (c_1 + c_2), \vphantom{\bigg(}\\
			c_4 = c_2 - \frac{\alpha N^{\beta}}{(2p-1)} (c_1 + c_2).
		\end{cases}
\]
\end{remark}

\begin{proof}[Proof of Lemma \ref{lem2}]
	We need to show that for any  local function $f$ such that 	$\int f d\bar\nu_{m(\cdot)} = 0$, we have
\[
	\int {\mathcal L}_N f (\omega) \,\bar\nu_{m(\cdot)} (d\omega) = 0.
\]
	Direct calculations show that the left hand side equals
	\begin{multline}
	\int \sum_{x\in\Z} \Big\{p g(\omega_x) \big( f (\omega^{x,x+1}) - f (\omega) \big) + (1-p)g(\omega_{x+1}) \big( f (\omega^{x+1,x}) - f (\omega) \big)\Big\} \,\bar\nu_{m(\cdot)} (d\omega)\\
	+ \int  \alpha N^\beta g(\omega_0) \big( f (\omega^{(0)}) - f (\omega) \big) \,\bar\nu_{m(\cdot)} (d\omega).\label{a3}
	\end{multline}
	Since, by change of variables,
	\begin{align*}
	m_{x+1}g(\omega_x) \,\bar\nu_{m(\cdot)} (d\omega) &= m_x g(\omega_{x+1}+1) \,\bar\nu_{m(\cdot)}  (d\omega^{x,x+1}),\\
	m_x g(\omega_{x+1}) \,\bar\nu_{m(\cdot)}(d\omega) &= m_{x+1} g(\omega_x+1) \,\bar\nu_{m(\cdot)} (d\omega^{x+1,x}),
	\end{align*}
	making the transformations $\omega^{x,x+1} \mapsto \omega$ and $\omega^{x+1,x} \mapsto \omega$ respectively, we  rewrite the first line of \eqref{a3} as
	\begin{align*}
	\int \sum_{x\in\Z} \bigg\{  \frac{(1-p) m_{x+1} - pm_{x}}{m_{x}} g(\omega_x) -&  \frac{(1-p) m_{x+1} - pm_{x}}{m_{x+1}} g(\omega_{x+1}) \bigg\} f (\omega) \,\bar\nu_{m(\cdot)}  (d\omega)\\
	= &\int \sum_{x \in \Z} \frac{(1-p) m_{x+1} + pm_{x-1}-m_{x}}{m_{x}} g(\omega_x) f (\omega) \,\bar\nu_{m(\cdot)} (d\omega)\\
	= &\int \alpha N^\beta g(\omega_0) f (\omega) \,\bar\nu_{m(\cdot)} (d\omega).
	\end{align*}
In the first identity above we have reindexed the second sum, whereas in the second identity, we have used that $m$ is a  solution to \eqref{a1}. Similarly, using the change of variable 

\[
	g(\omega_0)\,\bar\nu_{m(\cdot)}  (d\omega) = m_0  \,\bar\nu_{m(\cdot)}  (d\omega^{(0)}),
\]
and using the fact that $f$ has mean $0$ w.r.t.~$\bar\nu_{m(\cdot)}$, the second line of \eqref{a3} equals
\[
\int \alpha N^\beta \Big(m_0 - g(\omega_0)\Big) f (\omega) \,\bar\nu_{m(\cdot)} (d\omega) = - \int \alpha N^\beta g(\omega_0) f (\omega) \,\bar\nu_{m(\cdot)} (d\omega).
\]
	Therefore \eqref{a3} vanishes. This proves the result.
\end{proof}

\subsection{Second-class particles.} 
\label{sec:secondclass} 
We now consider  an auxiliary process $(\zeta(t))_{t\geqslant 0}$ composed of \emph{second-class particles}. Instead of killing a particle at the origin, we turn the killed particle into a \emph{second-class particle}, as explained below. The process $(\omega (t), \,\zeta (t))_{t \geq 0}$ evolves as follows:
\begin{itemize}
	\item at site $x\neq 0$, an $\omega$-particle jumps at rate $N g (\omega_x)$, and a $\zeta$-particle jumps at rate \[N (g (\omega_x + \zeta_x) - g (\omega_{x}))\; ;\]
	
	\item when a particle jumps, it jumps to the right with probability $p$ and to the left w.p.~$1-p$;
	
	\item at site $x=0$, an $\omega$-particle is transformed into a $\zeta$-particle at rate $\alpha N^{1+\beta} g (\omega_0)$.
\end{itemize}
It is easy to see that the process $(\omega (t) + \zeta (t))_{t \geq 0}$ evolves as the usual zero-range process according to the generator $N \Lex$, and $(\omega (t))_{t \geq 0}$ according to the generator $N \mathcal{L}_N$. Denote by $\overline{ \mathbb{P}}_{\mu_N}$ the law of the process $(\omega (t),\zeta (t))$, with no second-class particle at the initial time, \textit{i.e.}~$(\omega(0),\zeta(0)) \sim \mu_N \otimes \delta_{{\emptyset}}$ where $\delta_{\emptyset}$ is the Dirac measure on the empty configuration. Denote  by $\overline{\bb{E}}_{\mu_N}$ the corresponding expectation.

\medskip

We first state a lemma to bound the number of second-class particles up to time $t$. The following lemma is sufficient to prove the hydrodynamic limit in the case $\beta < 0$.

\begin{lemma}\label{lem:secondclass}
Assume that the initial distribution $\mu_N$ is stochastically dominated by $\nu_{\rho^\star}\,$ for some density $\rho^\star\in(0,+\infty)$.  Let $K_t:=\sum_{x\in\Z}\zeta_x(t)$ be the number of second-class particles up to time $t$.  Then  there exists a finite constant $C>0 $ independent of $N$ such that 
\[\overline{\bb{E}}_{\mu_N} [K_t] \leq C t\min \{\alpha N^{1+\beta}, \,N\}.\]
\end{lemma}

\begin{proof}
Since $K_t - \int_0^t \alpha N^{1+\beta} g (\omega_0 (s))\,ds$ is a mean zero martingale and since the process $(\omega(t))_{t\geqslant 0}$ is attractive and stochastically dominated by $\nu_{\rho^\star}$, 
\[
		\overline{\bb{E}}_{\mu_N} [K_t]  \leq \alpha N^{1+\beta} t {\rm E}_{\nu_{\rho^\star}} [g (\omega_0)] \leq a_0   \rho^\star \alpha N^{1+\beta} t.
\]
where $a_0$ is the Lipschitz constant of $g$, introduced in \eqref{eq:lipg}, which satisfies $g(k)\leq a_0 k$. 
We now show that $\overline{\bb{E}}_{\mu_N} [K_t] \leq C N t$ for some constant $C > 0$. Observe that $K_t$ is bounded by the number of particles visiting the origin up to time $t$. Let $Y_t$ be a homogeneous Poisson point process on $ \R$ with intensity $a_0 N$. Then the displacement of a typical first-class particle up to time $t$ is bounded stochastically by $Y_t$.  Dividing $\bb{Z}$ into the unions of $[j a_0 N t, (j+1) a_0 N t], \, j \in \bb{Z}$, we have 
\[
	\overline{\bb{E}}_{\mu_N}\, [K_t] \leq  \sum_{j = 3}^\infty 2 a_0 \rho^\star N t \overline{\bb{P}}_{\mu_N} [Y_t > j a_0 N t] + 6 a_0 \rho^\star N t.
\]
Above, the first term on the right hand side comes from the total number of particles which are initially at $\Z \backslash [-3a_0Nt,3a_0Nt]$ and visit the origin before time $t$, and the second term from the total initial number of particles in the interval $[-3a_0Nt,3a_0Nt]$.   Since $\overline{\bb{P}}_{\mu_N} [Y_t > j a_0 N t]  \leq \exp \{- a_0 (j-2) N t\}$, the first term on the right hand side vanishes as $N \rightarrow \infty$. This is enough to complete the proof.
\end{proof}

The following lemma will be used to identify the limiting density profile  to the left of the origin in the case $\beta \geq 0$.

\begin{lemma}\label{lem:NumberSec}
For any $t \geq  0$ and for any $\beta\in \R$,
	\begin{equation}
		\lim_{N \rightarrow \infty} \overline{\bb{E}}_{\mu_N}\, \bigg[ \frac{1}{N} \sum_{x \leq 0} \zeta_x (t)\bigg] = 0.
	\end{equation}
\end{lemma}

\begin{proof}
We only need to prove 
\[
	\lim_{N \rightarrow \infty} \overline{\bb{E}}_{\mu_N}\, \bigg[ \frac{1}{N} \sum_{x \leq - \sqrt{N}} \zeta_x (t)\bigg] = 0
\]
since by attractiveness of the process, we have 
\[\overline{\bb{E}}_{\mu_N}\, \bigg[ \frac1N \sum_{ -\sqrt{N} \leq x \leq 0} \zeta_x (t)\bigg]  \leq \rho^\star N^{-1/2}.\] Recall that $K_t$ is the number of second-class particles at time $t$.
Denote the positions of the second-class particles by $X_ i(t),$ for $ 1 \leq i \leq K_t$. Then \[\frac1N \sum_{ x \leq -\sqrt{N}} \zeta_x (t) =\frac1N \sum_{i=1}^{K_t} \mathbb{1} \{X_i (t) \leq - \sqrt{N}\}.\] Let $S_k$ be a discrete time simple random walk which jumps to the right with probability $p$ and to the left with probability  $1-p$. Then

\begin{align*}
\overline{\bb{P}}_{\mu_N} \big(X_1 (t) \leq - \sqrt{N}\big)  &= \sum_{j = 0}^\infty \overline{\bb{P}}_{\mu_N}\, (X_1 (t) \text{\, jumps $j$ times up to time $t$}) \times \overline{\bb{P}}_{\mu_N} (S_j \leq - \sqrt{N}) \\
&\leq \sup_{j \geq \sqrt{N}} \overline{\bb{P}}_{\mu_N} (S_j \leq - \sqrt{N}) \leq \sup_{j \geq \sqrt{N}}  \frac{j}{((2p-1) j+\sqrt{N})^2} \leq C N^{-1/2}
\end{align*}
for some finite constant $C$ independent of $N$. Together with Lemma \ref{lem:secondclass}, we have

\begin{align*}
\overline{\bb{E}}_{\mu_N} \bigg[\frac{1}{N} \sum_{i=1}^{K_t} \mathbb{1} \{X_i (t) \leq - \sqrt{N}\} \bigg] \leq  C  N^{-3/2} \overline{\bb{E}}_{\mu_N} [K_t] \leq C  N^{-1/2}.
\end{align*}
This completes the proof.
\end{proof}

 Since the process $(\omega (t) + \zeta (t))_{t\geqslant 0}$ is the usual asymmetric attractive zero-range process, we conclude this paragraph with a well-known result concerning the hydrodynamic limit of  $(\omega (t) + \zeta (t))_{t\geqslant 0}$, which has been proved in \cite{rezakhanlou91}.

\begin{theorem}[{\cite[Theorem 1.3]{rezakhanlou91}}]
\label{thm:HydroZrp}
Suppose the initial density profile $\rho_0$ is bounded. For any compactly supported and  continuous function $H:\R\to\R$, for any $t \geq 0$,  $\delta > 0$, $\beta\in\R$,
\[
	\lim_{N \rightarrow \infty}  \overline{\bb{P}}_{\mu_N} \pa{ \bigg|\frac{1}{N}\sum_{x \in \Z} H \left(\tfrac{x}{N}\right)( \omega_x (t) + \zeta_x (t)) - \int_{\R} H(u) \rho (t,u)\,du \bigg|> \delta }= 0,
\]
where $\rho (t,u)$ is the unique  entropy solution to \eqref{hEneg}.
\end{theorem}

\subsection{ Microscopic entropy inequality.}
\label{sebsection:coupling}  
In this subsection, we shall couple two copies of our Markov process, denoted by $\omega(t)$ and $\varpi (t)$. Recall that the process is supposed to be attractive (Assumption \ref{ass:H1}). 
If both processes have the same number of particles on a given site, a particle at this site will perform at the normal rate a particle jump/destruction, identical in both processes. If the processes do not have the same number of particles at a given site, at the minimal rate $g (\omega_x) \wedge g (\varpi_x)$, a particle will perform the same jump/destruction in both processes. To make up for the loss of jump rate, the process with the highest number of particles will make a particle jump/destruction while the other one will remain unchanged at a compensated rate $|g (\varpi_x) - g (\omega_x)|$. The generator of the coupled process  is given by its action on local Lipschitz functions $f:\Omega\times\Omega \to \R$ by

\begin{equation*}
	\begin{aligned}
		\tilde{\mathcal{L}}_N f (\omega,\varpi)&= \sum_{x \in \Z} \big(g (\omega_x) \wedge g (\varpi_x)\big)\;  \big\{  p f(\omega^{x,x+1},\varpi^{x,x+1}) + (1-p) f(\omega^{x,x-1},\varpi^{x,x-1}) - f (\omega,\varpi)\big\}\\
		&+ \sum_{x \in \Z} \big(g (\omega_x) - g (\omega_x) \wedge g (\varpi_x)\big) \big\{  p f(\omega^{x,x+1},\varpi) + (1-p) f(\omega^{x,x-1},\varpi) - f (\omega,\varpi)\big\}\\
		&+ \sum_{x \in \Z} \big(g (\varpi_x) - g (\omega_x) \wedge g (\varpi_x)\big) \big\{  p f(\omega,\varpi^{x,x+1}) + (1-p) f(\omega,\varpi^{x,x-1}) - f (\omega,\varpi)\big\}\\
		&+ \alpha N^{\beta} \big(g (\omega_0) \wedge g (\varpi_0)\big)\; \big\{ f(\omega^{(0)},\varpi^{(0)}) - f(\omega,\varpi)\big\}\\
		&+ \alpha N^{\beta} \big(g (\omega_0) - g (\omega_0) \wedge g (\varpi_0)\big)\; \big\{ f(\omega^{(0)},\varpi) - f(\omega,\varpi)\big\}\\
		&+ \alpha N^{\beta} \big(g (\varpi_0) - g (\omega_0) \wedge g (\varpi_0)\big)\; \big\{ f(\omega,\varpi^{(0)} )- f(\omega,\varpi)\big\}.
	\end{aligned}
\end{equation*}
One can easily check that under this coupling, if the process $(\omega, \varpi)$ evolves according to $\tilde{\mathcal{L}}_N$, then both marginal processes $\omega(t)$ and $\varpi (t)$ evolve according to $\mathcal{L}_N$. For a probability measure $\tilde{\mu}_N$ on $\Omega \times \Omega$, let $\tilde{\bb{P}}_{\tilde{\mu}_N}$ be the distribution on $D ([0,T], \Omega \times \Omega)$ of the process $(\omega(t),\varpi(t))$ with generator $N \tilde{\mathcal{L}}_N$ and with initial distribution $\tilde{\mu}_N$.

\medskip

For $x\in\Z$ and a positive integer $\ell$,  denote by $\omega^\ell_x= (2 \ell +1)^{-1} \sum_{|y-x| \leq \ell} \omega_y$ the average number of particles in a box of size $2\ell+1$ around $x$.
Given the coupling above, we now state a microscopic version of entropy inequality.

\begin{lemma}[Microscopic entropy inequality]\label{lem:microEntIneqn}
	Fix $\beta < 1$.  Let $\tilde{\mu}_N$ be a probability measure on $\Omega \times \Omega$ with both marginals stochastically dominated by $\nu_{\rho^\star}$ for some $\rho^\star > 0$. Then for any $\varepsilon > 0$ and for any non-negative smooth function $H$ with compact support in $(0,\infty) \times \bb{R}$,
	\begin{equation}\label{MicroEntInePos}
		\begin{aligned}
			\lim_{\ell \rightarrow \infty}\, \liminf_{N \rightarrow \infty}\, \tilde{\mathbb{P}}_{\tilde{\mu}_N}\,  \bigg[ \int_0^\infty  N^{-1} \sum_{x \in \Z} &\Big\{ \partial_t H (t,\tfrac x N) \big|\omega_x^\ell (t)- \varpi_x^\ell (t)\big| \\
			&+ (2p-1) \partial_u H (t,\tfrac x N)\; \big|\Phi (\omega_x^\ell(t))- \Phi(\varpi_x^\ell (t))\big|\Big\} dt \geq - \varepsilon \bigg] = 1.
		\end{aligned}
	\end{equation}
\end{lemma}

\begin{remark}
	Using the same argument, we could prove
	\begin{equation}\label{MicroEntIne0}
		\begin{aligned}
			\lim_{\ell \rightarrow \infty}\, \liminf_{N \rightarrow \infty}\, \tilde{\mathbb{P}}_{\tilde{\mu}_N}\,  \bigg[ \int_0^\infty  N^{-1} \sum_{x \in \Z} &\Big\{ \partial_t H (t,\tfrac x N) \big(\omega_x^\ell (t)- \varpi_x^\ell (t)\big)^{\pm} \\
			&+ (2p-1) \partial_u H (t,\tfrac x N)\; \big( \Phi (\omega_x^\ell(t))- \Phi(\varpi_x^\ell (t))\big)^{\pm} \Big\} dt \geq - \varepsilon \bigg] = 1.
		\end{aligned}
	\end{equation}
We will use \eqref{MicroEntInePos} in the case $\beta > 0$, and \eqref{MicroEntIne0} in the case $\beta = 0$. 
Note that the case $\beta<0$ is a direct consequence of the hydrodynamic limit stated in  Theorem \ref{thm:HydroZrp}, because then the particle destruction at the origin is too weak to affect the macroscopic density.
\end{remark}

\begin{proof}[Proof of Lemma \ref{lem:microEntIneqn}]
	We follow the strategy presented in \cite[Chapter 8]{klscaling} and \cite[Proposition 5.1]{landim1996hydrodynamical}.  For this reason we only sketch the proof here.  The main idea is as follows: we first investigate the total number of discrepancies over large boxes (of size of order $N$) to show that the two processes $\omega$ and $\varpi$ are ultimately locally ordered in such large boxes. This then permits to replace the average of the absolute value of relevant microscopic observables by the absolute value of their average. The result then follows from the one-block estimate which is stated in Lemma \ref{lemOneBlock} below.
	
	\medskip
	
	{\it Step 1.} The first step is to prove that in the limit $N \uparrow \infty$ the configurations $\omega$ and $\varpi$ are ultimately locally ordered. To be precise, for any $A > 0$, any $t > 0$, and any fixed $y \in \mathbb{Z}$, we have
	\begin{equation}\label{ordered}
		\lim_{N \rightarrow \infty} \, \tilde{\bb{E}}_{\tilde{\mu}_N} \,\bigg[ \int_0^t N^{-1} \sum_{|x| \leq A N} G_{x,x+y} (\omega(s),\varpi(s))\, ds\bigg] = 0,
	\end{equation}
	where $G_{x,x'} (\omega,\varpi) = \mathbb{1} \{\omega_x < \varpi_x, \omega_{x'} > \varpi_{x'}\}+ \mathbb{1} \{\omega_x > \varpi_x, \omega_{x'} < \varpi_{x'}\}$.  	 The idea is to investigate the mean zero martingale defined as
\[
	M^N_t := \frac 1N \sum_{|x| \leq A N } |\omega_x (t) - \varpi_x (t)| - \frac 1N\sum_{|x| \leq A N} |\omega_x (0) - \varpi_x (0)|  - 
	\int_0^t\, \sum_{|x| \leq A N}  \tilde{\mathcal{L}}_N \big(|\omega_x (s) - \varpi_x (s)|\big)\,ds,
	\]
	whose quadratic variation is bounded by a constant multiple of
	\[\int_0^t \Big\{ \frac{1}{N} \sum_{|x| \leq A N +1} \big[g(\omega_x (s)) + g (\varpi_x (s))\big]  + \alpha N^{\beta-1} [g(\omega_0 (s)) + g (\varpi_0 (s))] \Big\}\,ds.  \]
	By the attractiveness \mm{assumption \eqref{ass:H1}}, the martingale and the first two terms appearing in the expression of the martingale $M^N_t$ divided by $N$ vanish in the limit $N \rightarrow \infty$. Since
	\begin{multline*}
	\tilde{\mathcal{L}}_N |\omega_x  - \varpi_x | \leq	p G_{x,x-1} (\omega,\varpi ) \big|g (\omega_{x-1} ) - g (\varpi_{x-1} ) \big| \\
		+ (1-p)G_{x,x+1} (\omega,\varpi ) \big|g (\omega_{x+1} ) - g (\varpi_{x+1} ) \big| \Big\} + \alpha N^{\beta} \big|g (\omega_{0} ) - g (\varpi_{0} )\big|,
	\end{multline*}
we have
	\begin{multline}
\lim_{N \rightarrow \infty}\,\tilde{\bb{E}}_{\tilde{\mu}_N} \, \bigg[ \int_0^t  N^{-1}  \sum_{|x| \leq A N} \Big\{p G_{x,x-1} (\omega(s),\varpi (s)) \big|g (\omega_{x-1} (s)) - g (\varpi_{x-1} (s)) \big| \\
		\qquad + (1-p)G_{x,x+1} (\omega(s),\varpi (s)) \big|g (\omega_{x+1} (s)) - g (\varpi_{x+1} (s)) \big| \Big\} + \alpha N^{\beta-1} \big|g (\omega_{0} (s)) - g (\varpi_{0} (s)) \big| \,ds \bigg]=0. \label{eq:step1ine}
	\end{multline}
 In order to prove \eqref{ordered}, one can then directly implement the same induction argument  as in the  proof of \cite[Lemma 8.2.2]{klscaling}: more precisely, if we introduce
\begin{align*}I_m(\omega,\varpi):=\frac{1}{N}\sum_{|x|\leqslant AN} \Big[ & p\mathbb{1}\{m=\omega(x-1)<\varpi(x-1),\; \omega(x)>\varpi(x)\}\\ & + (1-p)\mathbb{1}\{m=\omega(x+1)<\varpi(x+1),\; \omega(x)>\varpi(x)\}\Big]\end{align*} then from \eqref{eq:step1ine} one deduces that, for any $m \in \N$, 
\[ \lim_{N\to\infty} \tilde{\bb E}_{\tilde\mu_N} \bigg[\int_0^t I_m(\omega(s),\varpi(s))ds\bigg]=0\]
and finally one is able to prove \eqref{ordered} by summing over $m$.
	
	\medskip
	
	{\it Step 2.}    For any non-negative smooth function $H$ with compact support in $(0,\infty) \times \bb{R}$, consider the mean zero martingale 
	\begin{equation}\label{martin1}
		\begin{aligned}
			\mathfrak{M}^{N}_t :=  \frac 1N \sum_{x \in \mathbb{Z}} H(t,\tfrac x N) |\omega_x (t) - \varpi_x (t)|
			-  \int_0^t (\partial_s + N \tilde{\mathcal{L}}_N) \bigg(\frac 1N \sum_{x \in \bb{Z}}    H (s,\tfrac x N) |\omega_x (s) - \varpi_x (s)|\bigg)\,ds.
		\end{aligned}
	\end{equation}
First, by Doob's inequality and direct calculations, 
	\begin{align*}
	 &\tilde{\bb{E}}_{\tilde{\mu}_N}\, [\sup_{0 \leq s \leq t} (\mathfrak{M}^N_s)^2] \leq 4 \tilde{\bb{E}}_{\tilde{\mu}_N}\, [ (\mathfrak{M}^N_t)^2]\\
		&\leq \int_0^t  \frac{4}{N} \sum_{x \in \Z} \tilde{\bb{E}}_{\tilde{\mu}_N} \left[\big|g(\omega_x (s)) - g(\varpi_x (s)) \big| G_{x,x\pm 1}(\omega_x (s),\varpi_x (s)) \right] \big[H(s,\tfrac{x}{N}) + H(s,\tfrac{x \pm 1}{N})\big]^2\,ds\\
		&\quad + \int_0^t \frac{4}{N} \sum_{x \in \Z}  \tilde{\bb{E}}_{\tilde{\mu}_N} \left[ \big|g(\omega_x (s)) - g(\varpi_x (s)) \big| (1-G_{x,x\pm 1}(\omega_x (s),\varpi_x (s))) \right] \big[H(s,\tfrac{x}{N}) - H(s,\tfrac{x \pm 1}{N})\big]^2\,ds\\
		&\quad + \int_0^t  4 \alpha N^{\beta-1} \tilde{\bb{E}}_{\tilde{\mu}_N} \left[ \big|g(\omega_0 (s)) - g(\varpi_0 (s)) \big|  \right] H(s,0)^2 ds.
	\end{align*} From Cauchy-Schwarz inequality, and using \eqref{ordered} and attractiveness \mm{assumption \eqref{ass:H1}}, the first term on the right side  converges to zero as $N \rightarrow \infty$. Also, the second term is bounded by $C N^{-2}$ and the third one by $C N^{\beta-1}$ for some constant $C=C(\alpha,H,\rho^\star)$. Since $\beta < 1$, we obtain that $\tilde{\bb{E}}_{\tilde{\mu}_N}\, [\sup_{0 \leq s \leq t} (\mathfrak{M}^N_s)^2] \to 0$ as $N\to\infty$.
	Moreover, since $H$ has compact support in $(0,\infty) \times \bb{R}$, we can take $t$ large enough for the first term on the  right hand side \mm{of \eqref{martin1}} to vanish. By direct calculations, we  bound the term $\sum_{x \in \bb{Z}}   H (s,\frac x N) \tilde{\mathcal{L}}_N  |\omega_x  - \varpi_x|$ from above by 	
\[
		\sum_{x \in \bb{Z}}   H(s,\tfrac x N) \big( p |g(\omega_{x-1}) - g (\varpi_{x-1}) | + (1-p) |g(\omega_{x+1}) - g (\varpi_{x+1}) | - |g(\omega_{x}) - g (\varpi_{x})|   \big).
\]
	An integration by parts allows us to write the first term above as
	\[
	\frac 1N	\sum_{x \in \bb{Z}}  (2p-1)  \partial_uH(s,\tfrac x N)  |g(\omega_{x}) - g (\varpi_{x})| + o_N (1).
	\]
	Therefore \mm{we conclude that}, for any $\epsilon > 0$,
	\begin{equation}\label{eqn1}
		\begin{aligned}
			\lim_{N \rightarrow \infty}\,\tilde{\bb{P}}_{\tilde{\mu}_N}\,  \bigg[ \int_0^\infty  N^{-1} \sum_{x \in \Z} &\Big\{ \partial_t H (t,\tfrac x N) \big|\omega_x (t)-\varpi_x (t)\big| \\
			&+ (2p-1) \partial_u H (t,\tfrac x N)\; \big|g(\omega_x (t))- g (\varpi_x (t))\big|\Big\} dt \geq - \varepsilon \bigg] = 1.
		\end{aligned}
	\end{equation}
	
	\medskip
	
	{\it Step 3.}  Due to the smoothness and compactness of $H$, in \eqref{eqn1} we can replace $|\omega_{x} - \varpi_{x}|$ and $|g(\omega_{x}) - g(\varpi_x)|$ with their spatial averages over sites in a box of size $2\ell +1$ around $x$.  Moreover, by \eqref{ordered} and repeating the proof of  \cite[Lemma 8.2.3]{klscaling}, we can replace  
	\[ (2\ell+1)^{-1} \sum_{|y-x| \leq \ell} |g(\omega_y (t))-g (\varpi_y (t))|  \quad \text{ resp. }\,(2\ell+1)^{-1} \sum_{|y-x| \leq \ell} |\omega_y (t)-\varpi_y (t)| \] 
	\textit{i.e.}~the average of the absolute value, by 
	\[\bigg|(2\ell+1)^{-1} \sum_{|y-x| \leq \ell} g(\omega_y (t))-g (\varpi_y (t))\bigg| \quad  \text{ resp. }\,|\omega^\ell_x (t) - \varpi^\ell_x (t)|\] 
	\textit{i.e.}~the absolute value of the average.  Therefore, we have
	\begin{equation}\label{aa}
		\begin{aligned}
			\lim_{N \rightarrow \infty}\,\tilde{\bb{P}}_{\tilde{\mu}_N}\,  \bigg(&\int_0^\infty \frac 1N \sum_{x \in \Z} \bigg\{ \partial_t H (t,\tfrac x N) \big|\omega_x^\ell (t)-\varpi_x^\ell (t)\big| 
			+ (2p-1) \partial_u H (t,\tfrac x N)\\
			&\qquad	\times \Big|(2\ell+1)^{-1} \sum_{|y-x| \leq \ell}g(\omega_y (t))- (2\ell+1)^{-1} \sum_{|y-x| \leq \ell} g (\varpi_y (t))\Big|\bigg\} dt \geq - \varepsilon \bigg) = 1.
		\end{aligned}
	\end{equation}
	Note that the probability in \eqref{aa} is bounded from above by
\begin{align*}
&\tilde{\bb{P}}_{\tilde{\mu}_N} \bigg( \int_0^t \frac 1N \sum_{|x| \leq A N} \tau_x V_\ell (\omega (s)) \,ds \geq \varepsilon \bigg) 
+  \tilde{\bb{P}}_{\tilde{\mu}_N} \bigg( \int_0^t \frac 1N \sum_{|x| \leq A N} \tau_x V_\ell (\varpi (s)) \,ds \geq \varepsilon \bigg)\\
&+ \tilde{\bb{P}}_{\tilde{\mu}_N}\,  \bigg( \int_0^\infty \frac 1N \sum_{x \in \Z} \bigg\{ \partial_t H (t,\tfrac x N) \big|\omega_x^\ell (t)-\varpi_x^\ell (t)\big| 
+ (2p-1) \partial_u H (t,\tfrac x N)\; \big|\Phi (\omega_x^\ell(t))- \Phi(\varpi_x^\ell (t))\big|\Big\} dt \geq - 3 \varepsilon \bigg),
\end{align*}
where \[
	V_\ell (\omega) = \Big| (2\ell+1)^{-1} \sum_{|y| \leq \ell} g (\omega_y) - \Phi (\omega^\ell_0) \Big|.
	\] Since $\varepsilon$ is arbitrary, we complete the proof of \eqref{MicroEntInePos} by using the  one-block estimate as stated in  Lemma \ref{lemOneBlock} below and \eqref{aa}.
 \end{proof}

For any $x \in \Z$, define the shift operator $\tau_x$, which acts on configurations $\omega \in \Omega$, as $(\tau_x \omega)_y = \omega_{x+y},\, y \in \Z$. Naturally, for local functions $f$ on $\Omega$, define $(\tau_x f) (\omega) = f(\tau_x \omega)$.

\begin{lemma}[One-block estimate]\label{lemOneBlock}
	If $\beta < 1$, then for any $\varepsilon > 0$, any $A > 0$ and any $t > 0$,
	\begin{equation*}
		\limsup_{\ell \rightarrow \infty}\, \limsup_{N \rightarrow \infty}\, \tilde{\bb{P}}_{\tilde{\mu}_N} \bigg( \int_0^t \frac 1N \sum_{|x| \leq A N} \tau_x V_\ell (\omega (s)) \,ds \geq \varepsilon \bigg) = 0,
	\end{equation*}
	where
	\[
	V_\ell (\omega) = \Big| (2\ell+1)^{-1} \sum_{|y| \leq \ell} g (\omega_y) - \Phi (\omega^\ell_0) \Big|.
	\]
	The same result obviously holds with $\omega (t)$ replaced by $\varpi (t)$.
\end{lemma}

The proof of the above lemma is quite standard, therefore we only sketch the proof here. We refer the readers to \cite[Chapter 5, Section 4]{klscaling} for its detailed implementation.

\begin{proof}
For two probability measures $\mu, \nu$ such that $\mu$ is absolutely continuous with respect to $\nu$, recall that the relative entropy of $\mu$ with respect to $\nu$ is defined as
\[H(\mu| \nu) = \int \log \frac{d\mu}{d \nu} d \mu.\]
Denote by $\mu_{N,t}$ the probability distribution of the accelerated process $\omega(t)$ at time $t$. Fix some $\rho > 0$ and denote $f^N_t = d \mu_{N,t} / d \nu_\rho$.

\medskip

{\it Step 1: Estimate of the Dirichlet form}. Following the proof in \cite[Chapter 5, Section 2]{klscaling}, one knows that the time evolution of the relative entropy is related to the Dirichlet form, as follows
\[H(\mu_{N,t} | \nu_\rho) + 2 N \int_0^t \E_{\nu_\rho} \big[ \sqrt{f^N_s} (\Lex + \Lsink) \sqrt{f^N_s} \big] d s \leq H(\mu_{N} | \nu_\rho).\]
Direct calculations show that
\[	\E_{\nu_\rho} \big[ \sqrt{f^N_s} \Lex  \sqrt{f^N_s} \big] = \frac{1}{4} \sum_{x \in \bb{Z}} 	\E_{\nu_\rho} \big[g(\omega_x) \big( \nabla_{x,x\pm 1} f^N_s (\omega)\big)^2\big] =: D^{\rm ZR}_N (f^N_s;\nu_\rho),\]
\ccl{where $\nabla_{x,x+1}f$ was defined after \eqref{lex}}. Moreover, \mm{there is a finite constant $C>0$ such that}\[\Big|\E_{\nu_\rho} \big[ \sqrt{f^N_s}  \Lsink \sqrt{f^N_s} \big]\Big| \leq C N^{\beta},\qquad \, H(\mu_{N} | \nu_\rho) \leq CN.\] Let \mm{us consider the time average} $\bar{f}^N_t = t^{-1} \int_0^t f^N_s ds$. By convexity of the Dirichlet form $D^{\rm ZR}_N$,
\[D^{\rm ZR}_N (\bar{f}^N_t;\nu_\rho) \leq C N^{\beta}.\]

\medskip

\red{{\it Step 2: Cutoff of large densities.}} By Markov's inequality and with the above notations, we  only need to prove  
\begin{equation*}
	\limsup_{\ell \rightarrow \infty}\, \limsup_{N \rightarrow \infty}\, \bb{E}_{\nu_\rho} \bigg[ \frac 1N \sum_{|x| \leq A N}  \tau_x V_\ell (\omega)  \bar{f}^N_t \bigg] = 0.
\end{equation*}
Then, by attractiveness \mm{assumption \eqref{ass:H1}}, we cut off large densities, which reduces the above equation to
\begin{equation}
\label{eq:Dir0}
	\limsup_{\ell \rightarrow \infty}\, \limsup_{N \rightarrow \infty}\, \sup_{D^{\rm ZR}_N (f;\nu_\rho) \leq C N^\beta}\,\bb{E}_{\nu_\rho} \bigg[ \frac 1N \sum_{|x| \leq A N}  \tau_x V_\ell (\omega)  f \mathbb{1} \{\omega^\ell_x \leq C\} \bigg] = 0.
\end{equation}

\medskip

\red{{\it Step 3: Estimate of the Dirichlet form over large microscopic boxes}. Given a probability density $f$ in the supremum above, consider its spatial average over the system, namely $\bar f:=\frac 1N \sum_{|x| \leq A N}\tau_x f$. Taking its conditional expectation $\bar f^{\ell}$ w.r.t.~the configuration outside of  $\Lambda_\ell:=[-\ell, \ell]$ yields that the expectation one wants to estimate in the left-hand side of \eqref{eq:Dir0} is simply 
\[\bb{E}_{\nu_\rho} \bigg[ V_\ell (\omega)  \mathbb{1} \{\omega^\ell_0 \leq C\} \bar f^\ell  \bigg]\] 
One can show that the Dirichlet form of $\bar f^\ell$ is small: roughly speaking, since the Dirichlet form of $f$ has total order $N^\beta$, it has order $N^{\beta-1}$ over each edge, thus order $\ell N^{\beta-1}$ over each box of length $\ell$. Letting $N \rightarrow \infty$, by convexity, the Dirichlet form over $\Lambda_\ell$ of the average $\bar f^\ell$ is less than that of $f$ over the same box, which vanishes, and we are left to prove}
\begin{equation*}
	\limsup_{\ell \rightarrow \infty}\sup_{D^{\rm ZR}_\ell (f;\nu_\rho)=0}\,\bb{E}_{\nu_\rho^\ell} \bigg[  V_\ell (\omega)  f \mathbb{1} \{\omega^\ell_0 \leq C\} \bigg] = 0.
\end{equation*}

\medskip

\red{{\it Step 4: Conclusion.}}
Above, $D^{\rm ZR}_\ell (f;\nu_\rho)$ is the corresponding Dirichlet form associated to the box $\Lambda_\ell:=[-\ell, \ell]$, and $\nu_\rho^\ell$ is the restriction of $\nu_\rho$ to $\N^{\Lambda_\ell}$. Last, we only need to decompose the above equation along hyperplanes with a fixed number of particles: \ccl{over those, since the Dirichlet form of $f$ vanishes, $f$ is not impacted by particle jumps ($\nabla_{x,x+1}f=0$ for $x\in \Lambda_\ell$), so that $f$ is constant over every hyperplane with a fixed number of particles. The strong law of large numbers then concludes the proof.}
\end{proof}

\section{Proof of Theorem  \ref{thmAttrac}}
\label{sec:proof}

In this section we give the proof of Theorem \ref{thmAttrac}.  We distinguish four cases, namely $\beta <0$, $\beta =0$, $0 < \beta < 1$ and $\beta \geqslant 1$.

\subsection{The case $\beta < 0$.} By Lemma \ref{lem:secondclass}, if $\beta < 0$, then for any compactly supported smooth function $H$,
\[
	\overline{\bb{E}}_N \Big[\frac{1}{N}\sum_{x \in \Z} H \left(\tfrac{x}{N}\right)\zeta_x (t) \Big] \leq  \frac{||H||_\infty}{N} \overline{\bb{E}}_N [K_t] \leq C  ||H||_\infty  \alpha t N^\beta \xrightarrow[N\to\infty]{}0,
\]
where $\zeta$ is the process of second-class particles introduced in Section \ref{sec:secondclass}. In particular, both processes $\omega (t)$ and $\omega (t) + \zeta (t)$ have the same macroscopic limit.  Theorem  \ref{thmAttrac} therefore follows from Theorem \ref{thm:HydroZrp}.

\subsection{The case $\beta = 0$} We now prove Theorem \ref{thmAttrac} in the case $\beta = 0$.  Several results proved in this subsection will also be used in the case  $\beta < 1$, we will make it explicit in the statements. We first introduce some notation.  Denote by $\mathcal{M}_+ (\bb{R})$ the space of positive Radon measures on $\bb{R}$ endowed with the vague topology. For each configuration $\omega\in\Omega$, define the empirical measure $\alpha^N \in \mathcal{M}_+ (\bb{R})$ as
\[\alpha^N (\omega) = \frac{1}{N} \sum_{x \in \bb{Z}} \omega_x \delta_{x/N},\]
where $\delta_{u}$ is the Dirac measure at $u\in\R$, and denote $\alpha^N_t = \alpha^N (\omega (t))$. Fix a time $T > 0$. Let $\Q_N$ be the distribution on the path space $D \big([0,T],\, \mathcal{M}_+ (\bb{R})\big)$ of the path $(\alpha^N_t)$, where $\alpha^N_0$ is distributed according to the pushforward distribution $\mu_N\circ (\alpha^N)^{-1}$. 

\begin{lemma}[Tightness]
\label{lem:tight}
	Suppose the initial measure $\mu_N$ is stochastically dominated by $\nu_{\rho^\star}$ for some $\rho^\star>0$. If $\beta < 1$, then the sequence $(\Q_N)_{N\in\N}$ is tight.  Moreover, all limit points are concentrated on weakly continuous paths which are absolutely continuous with respect to the Lebesgue measure and with density bounded by $\rho^\star$.
\end{lemma} 
 \begin{proof}[Proof of Lemma \ref{lem:tight}]
The proof of tightness is quite standard, see \cite[Lemma  4.1]{rezakhanlou91} for example. To prove Lemma \ref{lem:tight}, it is enough to show that for any $H \in C_c (\bb{R})$,
\[
	\lim_{\delta \rightarrow 0} \limsup_{N \rightarrow \infty} \bb{E}_{\mu_N} \bigg[ \sup_{|t-s| \leq \delta} \Big| N^{-1} \sum_{x \in \bb{Z}} H\big(\tfrac x N\big) \big(\omega_x (t) - \omega_x (s)\big)  \Big|\bigg] = 0.
\]
The only difference with \cite{rezakhanlou91} is that Dynkin's formula has an extra term coming from the destruction part of the dynamics, so that we only detail how the Glauber contribution is controlled, and refer to \cite{rezakhanlou91} for the other estimates.  For any function $f$, 
	\begin{equation}\label{eq:dynk}
		f(\omega(t))=f(\omega(0)) + \int_0^t N{\mathcal L}_N f(\omega(s))ds + \mathfrak{M}_t,
	\end{equation} where $\mathfrak{M}_t$ is a martingale whose quadratic variation is given by
	\begin{equation}
		\bb{E}_{\mu_N} \big[ \mathfrak{M}_t^2\big] = \bb{E}_{\mu_N}\bigg[\int_0^t N \Big({\mathcal L}_N f^2-2f  {\mathcal L}_N f\Big)(\omega(s))ds  \bigg].  \label{eq:QV}
	\end{equation}
	Choosing $f(\omega)=N^{-1}\sum H(\frac x N)\omega_x$, in addition to the proof performed in \cite{rezakhanlou91} we have two extra terms to control: 
	\begin{enumerate}
		\item[$(i)$] 
%
		first, in \eqref{eq:dynk}, the contribution of the Glauber dynamics is given by
\begin{align*}
\int_0^t N\Lsink f(\omega(s))ds&=  \int_0^t H(0) \Lsink  \omega_{0} (s)\,ds \\
&= \frac{H(0)}{N}( \omega_{0} (t) - \omega_{0} (0)) -  H(0) \int_0^t  \Lex \omega_{0} (s) \,ds  + \mathfrak{N}_{0,t},
\end{align*}
where $\mathfrak{N}_{0,t}$ is a martingale whose quadratic variation \lj{is given by the right side of \eqref{eq:QV} with $f (\omega) = N^{-1} H(0) \omega_0$. Direct calculations show that
\begin{align*}
\E_{\mu_N} \big[ \mathfrak{N}_{0,t}^2 \big]&	\leq  H(0)^2 \Big\{ \alpha N^{\beta-1} \int_0^t \E_{\mu_N} \big[g(\omega_0 (s)\big]\,ds \\
&\qquad\qquad \; + N^{-1} \int_0^t \E_{\mu_N} \big[pg(\omega_1 (s)) + (1-p) g (\omega_{-1} (s)) + g(\omega_0 (s))\big]\,ds\Big\} \\&\leq C \Big(N^{\beta - 1} + N^{-1}\Big)
\end{align*}
for some constant $C = C(\alpha,H,\rho^\star)$.} All terms but the third term on the right hand side vanish in $L^2(\bb{P}_N)$. To control the third term, since 
			\[\Lex \omega_0 = p g (\omega_{-1}) + (1-p) g (\omega_1) - g(\omega_0),\]
			and the function $g$ is Lipschitz,  it remains to prove
			\begin{equation}\label{tight1}
				\lim_{\delta \rightarrow 0} \limsup_{N \rightarrow \infty} \bb{E}_{\mu_N} \bigg[ \sup_{|t-s| \leq \delta} \bigg| \int_s^t \omega_x (\tau) d \tau  \bigg|\bigg] = 0.
			\end{equation}
			for any $x \in \bb{Z}$.  Using the  Cauchy-Schwarz inequality twice, the expectation above is bounded by
\[
				\delta^{1/2} \times \left(\bb{E}_{\mu_N} \bigg[  \int_0^T \omega_x^2 (\tau) d \tau  \bigg] \right)^{1/2}.
\]
			Since the initial measure is bounded by $\nu_{\rho^\star}$ and the process is attractive, the quantity above is bounded by
			\lj{\[ \delta^{1/2} T^{1/2} \left(\bb{E}_{\nu_{\rho^\star}}[ \omega_x^2 ] \right)^{1/2}.\]}
			This concludes the proof of \eqref{tight1}.
		
		\item[$(ii)$] second, in the integral in \eqref{eq:QV} appears the Glauber contribution 
		\lj{\[ 
		\bb{E}_{\mu_N}\bigg[\int_0^t \Big(N \Lsink f^2-2Nf \Lsink f\Big)(\omega(s))ds   \bigg]  = H(0)^2  \alpha N^{\beta-1} \int_0^t \E_{\mu_N} [g(\omega_0 (s))]\,ds \leq C N^{\beta - 1} 
		\] 
		for some constant $C = C(\alpha,H,\rho^\star)$, which vanishes in the limit as $N \rightarrow \infty$ since $\beta<1$.}
	\end{enumerate}
\end{proof}

We now characterize the macroscopic density profile to the left of the origin.

\begin{lemma}[Hydrodynamic limit to the left of the origin]\label{lem:0left}
	Assume $\beta < 1$. For any continuous function $H$ with compact support included in $(-\infty,0]$, and for any $t \geq 0$, $\delta > 0$, 
\[
		\lim_{N \rightarrow \infty}  \bb{P}_{\mu_N} \pa{\bigg| \frac{1}{N}\sum_{x \in \Z} H \left(\tfrac{x}{N}\right)\omega_x (t)  - \int_{\R} H(u) \rho (t,u)\,du \bigg| > \delta }= 0,
\]
	where $\rho(t,u)=\rho_L(t,u)$ is defined in Theorem \ref{thmAttrac} as the unique entropy solution to \eqref{hEneg}.
\end{lemma}

\begin{proof}
Recall the second-class particle's process introduced in Section \ref{sec:secondclass}.
We start by proving 
	\begin{equation}
	\label{eq:Hvanishes}
	\lim_{N \rightarrow \infty}  \overline{\bb{E}}_{\mu_N} \bigg[\frac{1}{N}\sum_{x \in \Z} |H \left(\tfrac{x}{N}\right)| \zeta_x (t) \bigg] =0.
	\end{equation}
	Since $H$ has its support included in $(-\infty,0]$, we bound 
	\[ \overline{\bb{E}}_{\mu_N} \bigg[\frac{1}{N}\sum_{x \in \Z} |H \left(\tfrac{x}{N}\right)| \zeta_x (t) \bigg]  \leqslant \|H\|_\infty \; \frac1N \sum_{x\leqslant 0} \overline{\bb{E}}_{\mu_N}  [\zeta_x(t) ]
	\] and the right hand side vanishes  from Lemma \ref{lem:NumberSec}. Note that the function $H$ is not necessarily continuous on $\R$, however by maximum principle (applied to the vanishing viscosity limit, see e.g. \cite[Lemma 4.46, p. 73]{malek1996}), and since $\omega$ is stochastically dominated by an equilibrium measure $\nu_{\rho^\star}$, $H$ can be approximated in $L^1(\R)$ by a sequence of smooth, compactly supported functions $H^\varepsilon$. We can now use straightforwardly Theorem \ref{thm:HydroZrp} and \eqref{eq:Hvanishes} to complete the proof.
\end{proof}

We now characterize the macroscopic density profile to the right of the origin.    Let $\mathcal{M}_+ (\bb{R}_+^2)$ be the space of positive Radon measures on $\bb{R}_+^2$ endowed with the vague topology.  For any integer $\ell$ and any configuration $\omega \in \Omega$, the \emph{Young measure} $\pi^{N,\ell} (du,d\lambda) \in \mathcal{M}_+ (\bb{R}_+^2)$  is defined by its action on continuous function $G: \bb{R}_+^2 \rightarrow \bb{R}$  with compact support as
\[\langle \pi^{N,\ell},G\rangle: = N^{-1} \sum_{x > \ell} G \big(\tfrac{x}{N},\omega^\ell_x\big).\]
Denote $\pi^{N,\ell}_t = \pi^{N,\ell} (\omega_t)$.  The mapping $\omega_t \mapsto (\alpha^N_t,\pi^{N,\ell}_t)$ induces a probability measure  $\Q_{N,\ell}$   on  the space $D ([0,T],\mathcal{M}_+ (\bb{R} )\times \mathcal{M}_+ (\bb{R}_+^2) )$. It is not hard to show that if $\beta < 1$, then the sequence of measures $\Q_{N,\ell}$ is tight (cf.~\textit{e.g.}~\cite[Lemma 8.1.2]{klscaling}). Denote by $\Q^*$ any limit point of the sequence $\Q_{N,\ell}$ as $N \rightarrow \infty$ then $\ell \rightarrow \infty$. To prove the hydrodynamic limit in our setting, the main step is to prove the following microscopic version of \eqref{hE0Integ1}. 

\begin{lemma}\label{lem:YoungMeasure0} Let $\beta = 0$. Assume the initial measure $\mu_N$ is stochastically  bounded by $\nu_{\rho^\star}$ for some $\rho^\star>0$.   There exists $M > 0$ such that for any non-negative test function $H \in C_K^{1,1} ((0,\infty) \times \R)$ and any constant $c \geq 0$,   $\Q^*$-almost surely,
		\begin{multline*}
			\int_0^\infty \int_0^\infty  \int_{0}^\infty  \Bigg\{ \partial_t H (t,u) (\lambda-c )^{\pm}  + (2p-1)  \partial_u H(t,u) (\Phi (\lambda) - \Phi(c) )^{\pm}\Bigg\}\pi_{t} (du,d \lambda) \,dt \\
+ M \int_0^\infty H(t,0) \big(\varrho(t)-c\big)^{\pm} \,dt \geq 0.
		\end{multline*}
\end{lemma}
\begin{proof}
	Fix $c \geq 0$.    We couple two processes $(\omega (t),\,\varpi (t))$ according to the generator $N \tilde{\mathcal{L}}_N$ as described in Subsection \ref{sebsection:coupling}.  Denote by $\tilde{\mu}_{N,i},\,i=1,2$, the $i$-th marginal of $\tilde{\mu}_{N}$. Take the first marginal $\tilde{\mu}_{N,1}$ of $\tilde{\mu}_{N}$ to be the initial measure $\mu_N$ defined in Subsection \ref{subsection:mainresult}, and the second marginal $\tilde{\mu}_{N,2}$ to be the invariant measure of the process  with generator $\mathcal{L}_N$,  with the four parameters $\{c_i\}_{1 \leq i \leq 4}$ as stated in Remark \ref{remark:invariantm} given by
	\[c_1 = -\frac{\alpha N^{\beta}}{2p-1} \Phi (c),\quad c_2 = \frac{\alpha N^{\beta}+2p-1}{2p-1} \Phi (c),\quad c_3 = 0,\quad c_4 = \Phi (c).\]
	It is easy to check 
	\[
	\lim_{N \rightarrow \infty} {\E}_{\tilde{\mu}_{N,2}} [\omega_{[Nu]}] = \begin{cases}
		R \big(\tfrac{2p-1}{\alpha + 2p-1} \Phi (c)\big) \quad&\text{if $u < 0$},\\
		c \quad&\text{if $u  > 0$},
	\end{cases}
	\]
where $R$  is the inverse function of  $\Phi$.
	Since the process $\varpi (t)$ is time invariant, by law of large numbers, for any $\varepsilon > 0$ and for any non-negative test function $H$ with compact support in $(0,+\infty) \times \bb{R}$, we have
	
	\begin{align*}
		&\lim_{\ell \rightarrow \infty}\, \limsup_{N \rightarrow \infty}\, \tilde{\mathbb{P}}_{\tilde{\mu}_N}\,  \bigg[ \int_0^\infty  \frac1N \sum_{x \leq 0}   H (t,\tfrac x N) \big| \varpi_x^\ell (t) - R \big(\tfrac{2p-1}{\alpha + 2p-1} \Phi (c) \big) \big| 
		dt \geq  \varepsilon \bigg]  = 0,\\
		&\lim_{\ell \rightarrow \infty}\, \limsup_{N \rightarrow \infty}\, \tilde{\mathbb{P}}_{\tilde{\mu}_N}\,  \bigg[ \int_0^\infty  \frac1N \sum_{x \geq 0}   H (t,\tfrac x N) | \varpi_x^\ell (t) -c|
		dt \geq  \varepsilon \bigg] = 0.
	\end{align*}
For any $\delta,\; \gamma > 0$, define  $H_{\gamma,\delta} (t,u) =  \psi_\gamma * \mathbb{1}\{\cdot > - \delta\}(u) H (t,u)  $, where $\psi_\gamma$ is an \emph{approximation of the identity}, so that $H_{\gamma,\delta} (t,u)$ converges as $\gamma \to 0$ to $H(t,u) \mathbb{1}\{u > - \delta\}$. Taking $H_{\gamma,\delta}$ as test function in \eqref{MicroEntIne0}, and letting $\gamma \rightarrow 0$, we have $\mathbb{Q}^*$-almost surely,
	\begin{align*}
		&\int_{0}^\infty \int_{0}^\infty  \int_0^\infty \bigg\{\partial_t H (t,u) \big(\lambda - c\big)^{\pm}   + (2p-1)\partial_u H (t,u)  \big(\Phi (\lambda) - \Phi (c)\big)^{\pm} \bigg\} \pi_{t} (du,d \lambda)  \,dt \\
		&+ 	\int_{0}^\infty \int_{-\delta}^0 \bigg\{\partial_t H (t,u) \Big( \rho_L (t,u) - R \big(\tfrac{2p-1}{\alpha + 2p-1} \Phi (c) \big)\Big)^{\pm}  + (2p-1)\partial_u H (t,u) \Big(\Phi (\rho_L (t,u)) - \tfrac{2p-1}{\alpha + 2p-1}\Phi (c) \Big)^{\pm} \bigg\} \,du\,dt \\
		&+ \int_{0}^\infty (2p-1) H (t,-\delta) \Big( \Phi (\rho_L (t,-\delta)) - \tfrac{\alpha+2p-1}{ 2p-1}\Phi (c)\Big)^{\pm} \,dt \geq 0.
	\end{align*}
	Since $\rho_L (t,u)$ is bounded, the second line above is bounded from above by $C \delta$ for some finite constant $C$. Furthermore, because $\Phi$ is increasing and Lipschitz continuous, $|\Phi (u) - \Phi (v)| \leq a_0 |u-v|$, we have
	\[\Big( \Phi (\rho_L (t,-\delta)) - \tfrac{\alpha + 2p-1}{ 2p-1}\Phi (c)\Big)^{\pm} \leq \frac{a_0 (\alpha+2p-1)}{2p-1} \Big(R \big(\tfrac{2p-1}{\alpha + 2p-1} \Phi (\rho_L (t,-\delta))\big) - c \Big)^{\pm}. \]
	We conclude the proof by letting $\delta \rightarrow 0$.
\end{proof}

\medskip

The rest of the proof is quite standard \cite[Section 5]{rezakhanlou91}, and we only sketch the proof here.  We first introduce the concept of \emph{measure-valued entropy solutions}. Let $\pi_{t,u} (d \lambda)$ be a measurable map from $\bb{R}_+^2$ into the space $\mc{M}_+ (\bb{R}_+)$. We say $\pi_{t,u} (d \lambda)$ is \emph{bounded} if there exists a finite constant $C$ such that $\pi_{t,u} (d \lambda)$ is supported in $[0,C]$ for a.e.~$(t,u) \in \R_+^2$.  We say $\pi_{t,u} (d \lambda)$ is a \emph{measure-valued entropy solution} to \eqref{hE0} if 	

\begin{itemize}
	\item there exists $M > 0$ such that for any non-negative test function $H \in C_K^{1,1} ((0,\infty) \times \R)$ and any constant $c \geq 0$,
\begin{multline}\label{mve_integral}
		\int_0^\infty \int_0^\infty  \bigg\{ \partial_t H (t,u) \int_0^\infty (\lambda-c )^{\pm} \pi_{t,u} (d \lambda) + (2p-1)\partial_u H(t,u) \int_{0}^\infty (\Phi (\lambda) - \Phi(c) )^{\pm}  \pi_{t,u} (d \lambda) \bigg\}  \,du\,dt \\
		+ M \int_0^\infty H (t,0) (\varrho (t)-c)^{\pm} \,dt \geq 0;
\end{multline}
	\item for every $A > 0$ 
	\begin{align}\label{mve_initial}
		\lim_{t \rightarrow 0} \int_0^A \int_{0}^\infty |\lambda - \rho_0 (u)| \pi_{t,u} (d \lambda) \,du = 0.
	\end{align}
\end{itemize}
It is obvious that if $\rho(t,u)$ is an entropy solution to \eqref{hE0}, then $\delta_{\rho(t,u)} (d\lambda)$ is a bounded measure-valued entropy solution. By \cite[Theorem 2.1]{bahadoran2012hydrodynamics}, the converse is also true: there exists a unique bounded measure-valued entropy solution $\pi_{t,u} (d \lambda)$ to \eqref{hE0}, which is of the form $\pi_{t,u} (d \lambda) =\delta_{\rho(t,u)} (d\lambda)$, $\rho(t,u)$ being the unique entropy solution.

\medskip

 We now state a lemma concerning the properties of the limiting measure $\Q^*$.
 
 \begin{lemma}\label{lem:Qstar}
 Assume the initial measure $\mu_N$ is stochastically  bounded by $\nu_{\rho^\star}$ for some $\rho^\star$.  For all $\beta \in \R$, we have $\Q^*$-almost surely,
 \begin{enumerate}[(i)]
 	\item $\alpha_t (du)$ and $\pi_t (du,d\lambda)$ are both  absolutely  continuous with respect to the Lebesgue measure $du$, whose densities are denoted by $\rho (t,u) $ and $\pi_{t,u} (d \lambda)$  respectively. Moreover, 
 	\[\int_0^\infty \lambda \pi_{t,u} (d \lambda) = \rho (t,u);\]
 	
 	\item $\pi_{t,u} (d \lambda)$ is bounded, with support on $[0,\rho^\star]$ for almost every $(t,u) \in \R_+^2$.
 \end{enumerate}
 \end{lemma}

\begin{proof}
Since the proof is the same to \cite[Lemma 5.5]{rezakhanlou91}, we only sketch it here for completeness.  
\begin{enumerate}\item[$(i)$] For non-negative continuous functions $H,G: \R_+ \rightarrow \mm{\R_+}$ with compact support,
\[\<\pi^{N,\ell}_t,HG\> = \frac{1}{N} \sum_{x > \ell} H(\tfrac{x}{N}) G (\omega_x^\ell (t)) \leq \|G\|_\infty  \frac{1}{N} \sum_{x > \ell} H(\tfrac{x}{N}).\]
Letting $N,\,\ell \rightarrow \infty$, one sees that the measure $\pi_t (du,d\lambda)$ is absolutely  continuous with respect to the Lebesgue measure $du$. Since for any continuous function  with compact support $H: \bb{R} \rightarrow \mm{\bb{R}_+}$, 
\[\langle \alpha^N_t, H \rangle    = \langle \pi^{N,\ell}_t, H \lambda \rangle\]
plus error terms that vanish uniformly as $N\to\infty$, one has
\[\int_0^\infty H(u) \alpha_t (du) = \int_0^\infty \int_0^\infty H(u) \lambda \pi_{t,u} (d \lambda) du.\]
Therefore,
\[\alpha_t (du) = \int_0^\infty \lambda \pi_{t,u} (d \lambda) du,\]
which implies the absolute continuity of $\alpha_t (du)$, and also the relation between $\rho(t,u)$ and $\pi_{t,u}$ in the statement of $(i)$.

\item[$(ii)$] For any $\varepsilon > 0$ and any non-negative continuous function $H: \R_+ \rightarrow \mm{\R_+}$ with compact support, by attractiveness, 
\[	\mathbb{E}_{\mu_{N}} \Big[ \<\pi^{N,\ell}_t,H \mathbb{1} \{\lambda > \rho^\star + \varepsilon \}\>\Big] \leq  \frac{1}{N} \sum_{x > \ell} H(\tfrac{x}{N}) \mathbb{P}_{\nu_{\rho^\star}} \Big[\omega_x^\ell (t) > \rho^\star + \varepsilon\Big].\]
We conclude the proof by letting $N,\ell \rightarrow \infty$ and using the law of large numbers. \end{enumerate} \end{proof}


\medskip

Now we are ready to prove Theorem \ref{thmAttrac} for $\beta = 0$.

\begin{proof}[Proof of Theorem \ref{thmAttrac} in the case $\beta = 0$.] 
We lay out the main arguments, but omit some technical details, for which we refer the reader to \cite[Lemma 5.2]{rezakhanlou91}. Fix a smooth function $H$ with compact support, that we decompose as $H(u) = H(u) \mathbb{1}\{u \geq  0\} + H(u) \mathbb{1}\{u <0\}$. The part corresponding to $H(u) \mathbb{1}\{u <0\}$ is treated in Lemma \ref{lem:0left}, so that we only need to prove that for any function $H$ which is continuous and compactly supported in $\R_+$, 
\[
		\lim_{N \rightarrow \infty}  \bb{P}_{\mu_N} \pa{\bigg| \frac{1}{N}\sum_{x \geq 0} H \left(\tfrac{x}{N}\right)\omega_x (t)  - \int_{0}^{+\infty} H(u) \rho (t,u)\,du \bigg| > \delta }= 0,
\]
where $\rho=\rho_R$ is the unique entropy solution to  \eqref{hE0} with  boundary condition given by  \eqref{eq:f}.
The proof is divided into several steps. 
	
	\medskip
	
\emph{Step 1.} We first assume the initial density profile $\rho_0$ restricted to $\R_+$ is a Lipschitz continuous function with compact support, so that we can use the proof of \cite[Lemma 5.6]{rezakhanlou91}. From the latter, it follows straightforwardly that the initial condition  \eqref{mve_initial}  holds $\Q^*$-almost surely. By Lemmas \ref{lem:YoungMeasure0} and \ref{lem:Qstar}, $\pi_{t,u} (d \lambda)$ is bounded and the inequality \eqref{mve_integral} holds $\Q^*$-almost surely. This implies the limit $\Q^*$ is concentrated on trajectories such that $\pi_{t,u} (d\lambda)$ is a bounded measure-valued entropy solution to \eqref{hE0}. Since there exists a unique bounded measure-valued entropy solution, which is given by a Dirac measure $\delta_{\rho(t,u)} (d\lambda)$, and by Lemma \ref{lem:Qstar}, $\Q^*$ is concentrated on trajectories such that $\rho(t,u)$ is entropy solution to \eqref{hE0}. By the uniqueness of the entropy solution, we conclude the proof for Lipschitz continuous  initial density profile $\rho_0$ with compact support.
	
	\medskip
	
\emph{Step 2.}  Now we assume the initial density profile $\rho_0$ is integrable on $\R_+$ and bounded.  On $\R_+$, we approximate $\rho_0$ by $\rho_{0,\epsilon} \in C_K^\infty (\R_+)$ such that
\[\lim_{\epsilon \rightarrow 0} \int_{\R} \big|(\rho_{0,\epsilon} - \rho_0 )(u)\big| du = 0,\]
and leave $\rho_{0,\epsilon}=\rho_0$ unchanged on $(-\infty,0)$.
We shall use the coupling introduced in Subsection \ref{sebsection:coupling}. We take the initial measure $\tilde{\mu}_N^\epsilon$ of the coupled process to be such that
\[\tilde{\mu}_{N,1}^\epsilon = \mu_N, \quad \tilde{\mu}_{N,2}^\epsilon (d \omega )= \bigotimes_{x\in\Z}\mathcal{P}_{\rho_{0,\epsilon}(x/N)}(d \omega_x)\] 
and that
\[\omega_x \leq \varpi_x \quad \text{if} \quad \rho_0 (\tfrac x N) \leq \rho_{0,\epsilon} (\tfrac x N), \quad \text{and} \quad  \omega_x \geq \varpi_x \quad \text{if} \quad \rho_0 (\tfrac x N) \geq \rho_{0,\epsilon} (\tfrac x N)\]
for all $x \in \Z$ and with probability one with respect to $\tilde{\mu}_N^\epsilon$.  Since under the coupling, the number of discrepancies only decreases in time, for any $t > 0$,
\[\lim_{N \rightarrow \infty} \tilde{\bb{E}}_{\tilde{\mu}_N^\epsilon} \bigg[ \frac{1}{N} \sum_{x \geq 0} \big|\omega_t (x) - \varpi_t (x)\big| \bigg] \leq 
\lim_{N \rightarrow \infty} \tilde{\bb{E}}_{\tilde{\mu}_N^\epsilon} \bigg[ \frac{1}{N} \sum_{x \in \Z} \big|\omega_0 (x) - \varpi_0 (x)\big| \bigg] \leq \int_0^{+\infty} \big|\rho_{0,\epsilon}(u) - \rho_0(u)\big|du,\]
which vanishes as $\varepsilon \to 0$.
Let $\tilde{\Q}_{N,\epsilon}$ be the law of the pair $(\alpha^N_t,\alpha^{N,\epsilon}_t)$ when $(\omega(t),\varpi (t))$ starts from $\tilde{\mu}_N^\epsilon$, where
\[\alpha^{N,\epsilon}_t = \frac{1}{N} \sum_{x \in \Z} \varpi_x (t) \delta_{x/N} (du). \] 
Let $\tilde{\Q}_{\epsilon}$  be any limit point of $\tilde{\Q}_{N,\epsilon}$ as $N \rightarrow \infty$, whose support is concentrated on absolutely continuous measures 
$(\alpha_t,\alpha^{\varepsilon}_t)=(\delta_{\rho(t,u)}du,\delta_{\rho_{\varepsilon}(t,u)}du)$. Passing to the limit in the last inequality, we have

\begin{equation}
\label{eq:rhorhoep}
\lim_{\varepsilon \to 0}\E_{\tilde{\Q}_{\epsilon}} \bigg[\int_{0}^\infty \big|\rho(t,u) - \rho_{\epsilon} (t,u)\big| du\bigg] \leq \lim_{\varepsilon \to 0}\int_{\R} \big|\rho_{0,\epsilon}(u) - \rho_0 (u)\big|du=0,
\end{equation}
where,  as a consequence of the first step  $\rho_{\epsilon} (t,u)$ is the unique entropy solution to \eqref{hE0} with initial datum $\rho_{0,\epsilon}$ and with boundary condition \eqref{eq:f} independent of $\varepsilon$. To justify the last statement, note that the boundary condition  \eqref{eq:f} only depends on the density left of the origin, which is initially independent from $\varepsilon$.  By the $L^1$-contraction principle (see \cite[Theorem 7.28]{malek1996} for instance) $\rho_{\epsilon}$ converges in $L^1$ to the solution of \eqref{hE0} with  boundary condition given by  \eqref{eq:f} and initial condition $\rho_0$. We conclude the proof by using \eqref{eq:rhorhoep} which proves that $\rho(t,u)$ must be the solution of \eqref{hE0} with  b.c. \eqref{eq:f}.

\medskip

\emph{Step 3.} \lj{ To extend the result to bounded  initial datum $\rho_0$, we follow \cite[Theorem 5.1]{rezakhanlou91} and use the approximated functions $\rho_{0,k} (u) = \rho_0 (u) \mathbb{1}\{u \leq k \}$ for $k > 0$.   We take the initial measure $\tilde{\mu}_{N,k}$ of the coupled process to be such that
\[\tilde{\mu}_{N,k,1} = \mu_N, \quad \tilde{\mu}_{N,k,2} (d \omega )= \bigotimes_{x\in\Z}\mathcal{P}_{\rho_{0,k}(x/N)}(d \omega_x)\] 
and such that $\omega_x = \varpi_x$ for $x\leq k N$. By \cite[Lemma 5.7]{rezakhanlou91}, for $k$ large enough, there exists a constant $c_0$ such that
\[\lim_{N \rightarrow \infty} \tilde{\bb{P}}_{\tilde{\mu}_{N,k}} \Big[ \omega_x (t) = \varpi_x (t)\; \text{ for any}\; x \leq (k - c_0 t) N\Big] = 1.\]
Similar to \emph{Step 2}, denote by $\tilde{\Q}_{N,k}$  the law of the pair $(\alpha^N_t,\alpha^{N,k}_t)$ when $(\omega(t),\varpi (t))$ starts from $\tilde{\mu}_{N,k}$, and by $\tilde{\Q}_{k}$  any limit point of $\tilde{\Q}_{N,k}$ as $N \rightarrow \infty$. Passing to the limit in the last expression,
\[\E_{\tilde{\Q}_{k}} \big[\rho(t,u) = \rho_k (t,u) \; \text{ for any}\; u \leq k - c_0 t\big] = 1.\]
In particular,
\[\E_{\tilde{\Q}_{k}} [\rho_R (t,u) = \rho_{k,R} (t,u) \; \text{ for any}\; 0 \leq u \leq k - c_0 t] = 1.\]
Since $\rho_{k,R} (t,u)$ converges to the entropy solution to \eqref{hE0} with initial value $\rho_0$ as $k \rightarrow \infty$, we conclude the proof by letting $k \rightarrow \infty$ in the case $\beta = 0$. }
\end{proof}

\subsection{The case $0 < \beta < 1$.}  In this subsection we prove Theorem \ref{thmAttrac} for the case $0 < \beta < 1$. Following the same steps presented in the last subsection, we only need to prove the following lemma.

\begin{lemma}
\label{lem:betaleq1}
Assume $0 < \beta <1$. Every limit point $\Q^*$  of the sequence $\Q_{N,\ell}$ is concentrated on paths $\pi_{t,u} (d\lambda)$ such that
\[
	\lim_{u \rightarrow 0} \int_0^{T} \int_{0}^{\infty}  \Phi (\lambda )\,\pi_{t,u} (d\lambda) \,dt= 0.
\]
\end{lemma}

\begin{proof}
To prove this result, we adapt the arguments laid out in \cite[Lemma 6.4]{landim1996hydrodynamical}. Fix $c \geq 0$.    We couple two processes $(\omega (t),\,\varpi (t))$ according to the generator $N \tilde{\mathcal{L}}_N$ as stated in Subsection \ref{sebsection:coupling}. Let $\tilde{\mu}_{N,1} = \mu_N$.  We take the second marginal $\tilde{\mu}_{N,2}$ of $\tilde{\mu}_N$ to be the invariant measure as stated in Remark \ref{remark:invariantm}, with the four parameters given by  
\[c_1 = - \frac{\alpha N^{\beta}}{2p-1+\alpha N^{\beta}} \Phi (c),\quad c_2 = \Phi (c),\quad c_3 = 0, \quad c_4 = \frac{2p-1}{2p-1 + \alpha N^{\beta}} \Phi (c).\] 
With the above choice,
\[
\lim_{N \rightarrow \infty} \E_{\tilde{\mu}_{N,2}} [\omega_{[Nu]}] = \begin{cases}
	c \quad&\text{if $u < 0$},\\
	0 \quad&\text{if $u  > 0$}.
\end{cases}
\]
Since the process $\varpi (t)$ is time invariant, by law of large numbers, for any $\varepsilon > 0$ and for any non-negative test function $H$ with compact support in $(0,\infty) \times \bb{R}$, we have

\begin{align*}
	&\lim_{\ell \rightarrow \infty}\, \limsup_{N \rightarrow \infty}\, \tilde{\mathbb{P}}_{\tilde{\mu}_N}\,  \bigg[ \int_0^\infty  N^{-1} \sum_{x \leq 0}   H (t,\tfrac x N) \big| \varpi_x^\ell (t) - c \big| 
dt \geq  \varepsilon \bigg]  = 0,\\
 &\lim_{\ell \rightarrow \infty}\, \limsup_{N \rightarrow \infty}\, \tilde{\mathbb{P}}_{\tilde{\mu}_N}\,  \bigg[ \int_0^\infty  N^{-1} \sum_{x \geq 0}   H (t,\tfrac x N) \varpi_x^\ell (t) 
dt \geq  \varepsilon \bigg] = 0.
\end{align*}
 By Lemma \ref{lem:microEntIneqn}, 
 
\begin{equation}\label{boundary2}
	\begin{aligned}
		&\int_0 ^\infty \, \int_0^\infty \, \Big\{ \partial_tH (t,u) \int_{0}^{\infty} \lambda  \pi_{t,u} (d\lambda) +  (2p-1) \partial_u H (t,u) \int_{0}^{\infty} \Phi(\lambda) \pi_{t,u} (d\lambda)  \Big\} \, du \,dt\\
		&+ \int_0^\infty \, \int_{- \infty}^0 \, \Big\{ \partial_t H (t,u)  |\rho_L (t,u) - c|  +  (2p-1) \partial_u H (t,u)  |\Phi(\rho_L (t,u)) - \Phi (c)|  \Big\} \, du \,dt \geq 0
	\end{aligned}
\end{equation}
with $\Q^*$-probability one. The rest of the proof of Lemma \ref{lem:betaleq1} is adapted from \cite[Theorem 4.1]{diperna1985measure}. In order not to burden with technical details, we now present a formal argument, and refer to the latter for additional technical details. From \eqref{boundary2}, we first obtain formally, for any $c$, $t$, $u$,

\begin{multline}
\label{eq:formalpi}
\<\partial_t \pi_{t,u} \mathbb{1}\{u \geq 0\}, \lambda\> + (2p-1)\<\partial_u \pi_{t,u} \mathbb{1}\{u \geq 0\}, \Phi (\lambda)\>  \\
+ \<\partial_t \pi_{t,u} \mathbb{1}\{u \leq 0\}, |\lambda - c|\> + (2p-1)\<\partial_u \pi_{t,u} \mathbb{1}\{u \leq 0\}, |\Phi (\lambda)-\Phi (c)|\>  \leq 0
\end{multline}
in the sense of distributions, where we define $\pi_{t,u} \mathbb{1}\{u \leq 0\} = \delta_{\rho_L (t,u)} \mathbb{1}\{u \leq 0\}$. Observe that
\begin{align*}
	0 &=  \partial_t \< \pi_{t,u} \mathbb{1}\{u \leq 0\}, |\lambda - \rho_L (t,u)|\> + (2p-1) \partial_u \< \pi_{t,u} \mathbb{1}\{u \leq 0\}, |\Phi (\lambda)-\Phi (\rho_L (t,u))|\> \\
	&=  \< \partial_t  \pi_{t,u} \mathbb{1}\{u \leq 0\}, |\lambda - \rho_L (t,u)|\> + (2p-1) \<\partial_u  \pi_{t,u} \mathbb{1}\{u \leq 0\}, |\Phi (\lambda)-\Phi (\rho_L (t,u))|\> \\
	&+  \< \pi_{t,u} \mathbb{1}\{u \leq 0\}, \partial_t |\lambda - \rho_L (t,u)|+ (2p-1) \partial_u   |\Phi (\lambda)-\Phi (\rho_L (t,u))| \> \\
	&\leq \< \partial_t  \pi_{t,u} \mathbb{1}\{u \leq 0\}, |\lambda - \rho_L (t,u)|\> + (2p-1) \<\partial_u  \pi_{t,u} \mathbb{1}\{u \leq 0\}, |\Phi (\lambda)-\Phi (\rho_L (t,u))|\>
\end{align*}
in the sense of distributions. In the last inequality, we use the fact that $\rho_L (t,u)$ satisfies the entropy inequality and thus the penultimate line is non-positive in the sense of distributions.  Therefore, choosing $c=\rho_L(t,u)$ in \eqref{eq:formalpi} yields
\[\<\partial_t \pi_{t,u} \mathbb{1}\{u \geq 0\}, \lambda\> + (2p-1)\<\partial_u \pi_{t,u} \mathbb{1}\{u \geq 0\}, \Phi (\lambda)\> \leq 0\]
in the sense of distributions.  More explicitly, this identity yields for any non-negative test function $H$ that
\begin{equation}\label{boundary1}
\begin{aligned}
	\int_0 ^\infty \, \int_0^\infty \, \bigg\{ \partial_t H (t,u) \int_{0}^{\infty} \lambda  \pi_{t,u} (d\lambda)
	 +  (2p-1) \partial_u H (t,u) \int_{0}^{\infty} \Phi(\lambda)  \pi_{t,u} (d\lambda)  \bigg\} \, du \,dt \geq 0
\end{aligned}
\end{equation}
with $\Q^*$-probability one.   \lj{We remark that  in the above formal argument we need to reformulate those inequalities involving distributions into inequalities in terms of integrals over test functions. Moreover, the constant $c$ does not depend on $t$ or $u$. This could be overcome by appropriately regularizing $\pi_{t,u}$ and $\rho_L (t,u)$.}

Fix $\delta > 0$. Using the same strategies presented in the proof of Lemma \ref{lem:YoungMeasure0},  we approximate $H (t,u) = \mathbb{1} \{ t \leq T\} \times \mathbb{1} \{ u \leq \delta\}$ by smooth functions with compact support, and  together with \eqref{boundary1}, we obtain that there exists $C > 0$ such that
\begin{equation}
	\int_0^{T} \, \int_{0}^{\infty} \,  \Phi(\lambda)  \pi_{t,\delta} (d \lambda)\,dt \leq C \delta
\end{equation}
with $\Q^*$-probability one. This completes the proof by letting $\delta \rightarrow 0$.
\end{proof}

\subsection{The case $\beta \geq 1$.}  

To indicate the dependence of the process on the parameter $\beta$, let $\omega^{(\beta)} (t)$ be the process with generator $N (\Lex + \Lsink)$. Without confusions and to make notations simple, in this subsection we denote  $\omega (t) := \omega^{(1/2)} (t)$ and $\eta(t) = \omega^{(\infty)} (t)$. In other words, in  $\eta$ any particle that reaches the origin is destroyed. We still denote by $\Prob_{\mu_N}$ the distribution of the coupled processes, and by $\E_{\mu_N}$ the corresponding expectation. We shall couple $\eta (t)$ and $\omega (t)$ together such that if $\omega (0) = \eta (0)$ at the initial time, then
\begin{equation}\label{vanishCoup}
\limsup_{N \rightarrow \infty} \frac{1}{N} \sum_{x \in \bb{Z}}  \bb{E}_{\mu_N} \Big[\big| \eta_x (t) - \omega_x (t) \big| \Big]= 0.
\end{equation}
This is enough to prove Theorem \ref{thmAttrac} for the case $\beta \geq 1$, since the case $\beta=1/2$ has already been proved. Indeed, since the process is attractive,  for any compactly supported and  continuous function $H:\R\to\R$ and for any $\beta \geq 1$,
\begin{align*}
\bb{E}_{\mu_N} \bigg[ \Big| \frac{1}{N} \sum_{x \in \bb{Z}}  H \big(\tfrac x N\big) (\omega^{(\beta)}_x (t) - \omega_x (t)) \Big| \bigg] &\leq \frac{||H||_\infty}{N} \sum_{x \in \bb{Z}}  \bb{E}_{\mu_N} \Big[\big| \omega^{(\beta)}_x (t) - \omega_x (t) \big| \Big]\\
 &\leq \frac{||H||_\infty}{N} \sum_{x \in \bb{Z}}  \bb{E}_{\mu_N} \Big[\big| \eta_x (t) - \omega_x (t) \big| \Big] \rightarrow 0 \quad \text{as $N \rightarrow \infty$.}
\end{align*}
We now describe the coupling.  We label the $\eta$-particles by $Y^{k,i} (t)$,  and the $\omega$-particles by  $Z^{k,i} (t)$, where $k \in \bb{Z},\, i = 1,2,\ldots,\eta_k (0)$. Let
\[
\eta_x (t) = \sum_{k \in \bb{Z}} \sum_{i=1}^{\eta_k (0)} \mathbb{1} \{Y^{k,i} (t) = x\}, \qquad
\omega_x (t) = \sum_{k \in \bb{Z}} \sum_{i=1}^{\eta_k (0)} \mathbb{1} \{Z^{k,i} (t) = x\}.
\]
If $Y^{k,i}$ and $Z^{k,i}$ are at the same site, then we say these two particles are coupled; otherwise they are uncoupled. We shall see that $\eta (t) \leq \omega (t)$ for all $t \geq 0$ and that $Y$-particles are always coupled with $Z$-particles. The dynamics is as follows:
\begin{itemize}
	\item at rate $N g (\eta_x (t))$, a pair of coupled particles is chosen uniformly from the coupled particles at site $x$, and jumps to site $x+1$ with probability  $p$, or to site $x-1$ with probability $1-p$;
	
	\item at rate $N (g (\omega_x (t)) -g (\eta_x (t)))$, an uncoupled $Z$-particle is chosen uniformly from the uncoupled $Z$-particles at site $x$, and jumps to site $x+1$ with probability  $p$, or to site $x-1$ with probability $1-p$;
	
	\item at the origin,  $Y$-particles die instantaneously and at rate $(N+ \alpha N^{3/2}) g (\omega_0 (t))$, a $Z$-particle is chosen uniformly from those particles at the origin, then it jumps to site $x+1$ with probability $p / (1+\alpha N^{1/2})$, or  to site  $x-1$ with probability $(1-p) / (1+\alpha N^{1/2})$, or dies with probability $\alpha N^{1/2} / (1+\alpha N^{1/2})$.
\end{itemize}
With the above coupling, we have
\[
\sum_{x \in \bb{Z}} \big| \eta_x (t) - \omega_x (t) \big|  = \sum_{k \in \bb{Z}} \sum_{i=1}^{\eta_k (0)} \mathbb{1}  \{ \text{$Z^{k,i}$ visits the origin before time $t$ and survives by time $t$} \}.
\]
Using the same argument as in the proof of Lemma \ref{lem:secondclass}, we can restrict the above summation to a large box $\{k: k \in [-AN,AN]\}$ for some $A > 0$. Depending on whether $Z^{k,i}$ is at the origin or not at time $t$, we can bound the expectation of the last expression by
\[
\bb{E}_{\mu_N} [\omega_0 (t)] + \sum_{k \in [-An,An]} \bb{E}_{\mu_N} [\eta_k (0)] \frac{1}{1+\alpha N^{1/2}},
\]
which divided by $N$ converges to zero as $N \rightarrow \infty$. This proves $\eqref{vanishCoup}$.

%
%
%
%
%

\section{The linear case}
\label{app}

In this section, we provide a simpler, alternative proof of the hydrodynamic limit (Theorem \ref{thmAttrac}) in the linear case where the jump rate is given by $g(k)=k$. Note that in this case, (see Remark \ref{rem:lincase}), the hydrodynamic limit follows a linear advection equation, and thus does not develop shock. Proving existence and uniqueness of weak solutions is then straightforward.
The alternative proof of Theorem \ref{thmAttrac} we propose in the linear case follows the \emph{duality approach} already used in \cite{Erignoux18, ELX18} for symmetric exclusion processes. However, it is significantly easier for the linear asymmetric zero-range process, because the estimation of the two-points correlations of the system is straightforward.

\medskip

For any sites $x,y\in\Z$, any $t\geq 0$,  we define the following discrete \emph{density field}
\begin{equation}
\label{eq:Defrhon}
\rho^N_x(t):=\E_{\mu_N}\big[\omega_x(t) \big], 
\end{equation}
as well as the \emph{two-points correlation field}
\begin{equation}
\label{eq:Defphi}
\varphi^N_{x,y}(t):=\E_{\mu_N}\big[(\omega_x(t)-\rho^N_x(t))(\omega_y(t)-\rho^N_y(t))\big].
\end{equation}
As detailed in \cite{Erignoux18}, the hydrodynamic limit in the linear case is a straightforward consequence of Lemmas \ref{lem:DensityField} and \ref{lem:CorrelationField} below.  The first lemma states that the discrete density field, once properly rescaled, converges to the hydrodynamic limit.

\begin{lemma}
\label{lem:DensityField}
The density field $\rho^N$ defined in \eqref{eq:Defrhon} converges as $N\to\infty$ towards the function $\rho$  defined in \eqref{eq:Defrho}, in the sense that: for any $t\geqslant 0$, any $u\in \R$,
	\begin{equation}\label{es2}
	\lim_{N \rightarrow \infty} \rho^N_{\lfloor N u\rfloor}(t) = \rho(t,u), 
	\end{equation}
where the convergence is  uniform in $u$ over compact subsets of $\R\setminus\{0,(2p-1)t\}$. In the identity above, the function $\rho$ is the linear hydrodynamic limit defined by \eqref{eq:Defrho}. In particular, for any  $H \in  C_K (\R)$, 

	\begin{equation*}
	\lim_{N \rightarrow \infty} \frac{1}{N} \sum_{x \in \mathbb{Z}} H \left(\tfrac{x}{N}\right) \rho^N_x(t) = \int_\R H(u) \rho (t,u) \;du.
	\end{equation*}
\end{lemma}

The second lemma states that the two-points correlation field vanishes away from the diagonal.
\begin{lemma}\label{lem:CorrelationField}
For any $\varepsilon>0$, and any $t\geqslant 0$, 
	\begin{equation}\label{escor2}
	\lim_{N \rightarrow \infty}\sup_{\substack{x,y\in \Z\\ |x-y|\geq \varepsilon N}}| \varphi^N_{x,y}(t)|=0.
	\end{equation}
\end{lemma}

We will not detail the proof that the hydrodynamic limit follows from these two estimates, we simply sketch the main steps and refer the interested reader to \cite[Section 5]{Erignoux18} for a detailed implementation. The main argument comes from the straightforward bound

\begin{multline*}
\E\bigg[\bigg(\frac{1}{N} \sum_{x \in \mathbb{Z}} H \left(\tfrac{x}{N}\right) \omega_x(t) - \int_\R H(u) \rho (t,u) \bigg)^2\bigg]\\
\leq \frac{2}{N^2}\sum_{x,y\in \Z} H \left(\tfrac{x}{N}\right) H \left(\tfrac{y}{N}\right) \varphi^N_{x,y}(t)+ 2 \bigg(\frac{1}{N} \sum_{x \in \mathbb{Z}} H \left(\tfrac{x}{N}\right) \rho^N_x(t) -\int_\R H(u) \rho (t,u) \;du\bigg)^2,
\end{multline*}
which is obtained by adding and subtracting $N^{-1} \sum_x H (x/N) \rho^N_x(t)$ inside the parenthesis. The second term in the right hand side vanishes according to Lemma \ref{lem:DensityField}. The first term can be split, for any $\varepsilon>0$ into a term that vanishes as $N\to\infty$ according to Lemma \ref{lem:CorrelationField}, and a second sum which contains $\mathcal{O}((\varepsilon N)^2)$ terms. Because the process is stochastically dominated by some equilibrium distribution $\nu_{\rho^\star}$, the $\varphi^N_{x,y}(t)$ are bounded, uniformly in $x,y, t$, by some constant $C(\rho^\star)$. In particular, each of the terms in the sum are of order $C(H, \rho^\star)/N^2$, and this last contribution vanishes as $N\to\infty$ and $\varepsilon \to 0$, which proves the hydrodynamic limit.

%
%

\begin{proof}[Proof of Lemma \ref{lem:DensityField}]
By Dynkin's formula, the values $\rho^N_x(t):=\rho^N(t,\frac x N)$ are the solutions of the following system: 
\begin{equation}\label{eq:b1}
\begin{cases}
\frac{d}{dt} \rho^N_x(t)  = N\Big((1-p) \rho^N_{x+1}(t) + p \rho^N_{x-1}(t)-\rho^N_x(t)\Big), & x \neq 0\\ 
\frac{d}{dt} \rho^N_0(t)  = N\Big((1-p) \rho^N_1(t) +p \rho^N_{-1}(t) -(1+\alpha N^{\beta}) \rho^N_0(t) \Big).
\end{cases} 
\end{equation}
The main ingredient to prove  Lemma \ref{lem:DensityField} is a duality relationship that expresses  the density field $\rho^N(t)$ as a function of an asymmetric random walk $X$ on $\Z$.
Define a cemetery state $\dd$, and consider a continuous-time random walk $X_t$ on $\Z$, killed at rate $ \alpha N^{1+\beta}$ at the origin, at which point it goes to the cemetery state $\dd$. The random walk is asymmetric, and jump at rate $Np$ to the left and $N(1-p)$ to the right, \textit{i.e.}, its drift is inverted w.r.t.~the particles in our zero-range process. Denote by $\mathcal{L}_N^\dagger$ the generator of this process, acting on local functions $f:\Z\cup\{\dd\}\to\R$ as
\begin{equation}
\label{eq:Ldag}
(\mathscr{L}_N^\dagger f)_x=Np\br{f_{x-1}-f_x}+N(1-p)\br{f_{x+1}-f_x}+\alpha N^{1+\beta}\br{f_{\dd} -f_x}{\mathbb{1}}\{x=0\}.
\end{equation}
We denote by $\tau\in[0,+\infty]$ the time at which the random walk is killed, namely $\tau=\inf\{ t>0, \; X_t=\dd\}$, by  $\Prob_x$ the distribution of $X_t$ started from $x$, and by ${\mathbb E}_x$ the corresponding expectation. For convenience, we let the random walk run after it is killed by letting it jump at infinite rate from $\mathfrak{d}$ to $0$, but keep track that it died.


With these notations,  \eqref{eq:b1} yields 
\begin{equation}
\label{eq:FCdens}
\begin{cases}
\frac{d}{dt} \rho^N_x(t) =\mathscr{L}_N^\dagger \rho^N_x(t)& \quad  x\in \Z,\; t>0,\\ 
\rho^N_{\dd}(t)=0& \quad t>0,\\
\rho^N_x(0)=\rho_0(\frac xN ) & \quad  x\in \Z,
\end{cases}
\end{equation}
so that by Feynman-Kac's formula, we have the dual identity 
\begin{equation}
\label{duality_mean}
\rho^N_x(t)=\mathbb{E}_x\cro{\rho_0\Big(\frac{X_t}N\Big)\;{\bb 1}\{\tau>t\}}.
\end{equation}
We now consider the behavior of the random walk $X_t$. Define $p_N,\; q_N,\; r_N$ as
\[p_N=\frac{\alpha N^\beta}{1+\alpha N^\beta},  \quad q_N=  \frac{2p-1}{1+\alpha N^\beta} \quad \mbox{ and } \quad r_N=\frac{2-2p}{1 + \alpha N^{\beta}}=1-p_N-q_N,\]
which are the probabilities that: the random walk $X$, initially at the origin, gets killed before jumping out ($p_N$), the random walk is not killed and escapes to infinity without ever coming back to the origin ($q_N$), and the random walk is not killed immediately  but eventually comes back to the origin ($r_N$).

From this, we deduce that the probability that, starting at the origin, the probability that $X_t$ gets killed before escaping to infinity is 
\begin{equation}\label{killprob}
\widetilde{\alpha}_N:=\mathbb{P}_0 (\tau < \infty) =  p_N \sum_{k \geq 0} r_N^k=\frac{\alpha N^\beta}{\alpha N^\beta+2p-1}\underset{N \to\infty}{\longrightarrow}\widetilde{\alpha}.
\end{equation}
where $\widetilde{\alpha}$ was defined in \eqref{eq:Defalphat}. 

\medskip

Fix now $u\in \R$, and assume that  $X_0=\lfloor Nu \rfloor$. Recall that we let the random walk $X$ jump asymmetrically after it dies, we further claim that 
	\begin{equation}
	\label{estimateXt0}
	\Prob_{\lfloor Nu \rfloor}\pa{\sup_{s\leq t}\Big|\frac{X_s}{N}- u+(2p-1)s\Big|\geq 2 N^{-1/4}}=\mathcal{O}(N^{-1/8}).
	\end{equation}
We define the martingale $M_s:=X_s- \lfloor Nu \rfloor+N(2p-1)s$, which is a martingale w.r.t.~the natural filtration of $X_\cdot$.  Fix $ t>0$, by Doob's inequality
	\begin{equation}
	\label{estimate1}
	\Prob_{\lfloor Nu \rfloor}\pa{\sup_{s\leq t}|M_s|\geq N^{3/4}}\leq {\mathbb E}_{\lfloor Nu \rfloor}\cro{|M_t|}N^{-3/4}
	\leq N^{-3/4}\left[1+  \int_{1}^\infty \Prob_{\lfloor N u\rfloor}(|M_t|\geq v)dv\right] .
	\end{equation}
	Fix $K>0$, we define for $k\leq N$ 
	\[Y_k=\big(M_{t(k+1)/N} - M_{tk/N}\big)\;{\bb 1}\big\{|M_{t(k+1)/N} - M_{tk/N}|\leq K\big\}.\] 
	By Azuma's inequality, 
	\[\Prob_{\lfloor N u\rfloor}\bigg(\Big|\sum_{k=0}^{N-1} Y_k\Big|\geq v\bigg)\leq 2 e^{-v^2/2NK^2}.\]
	In particular, 
	\[\Prob_{\lfloor N u\rfloor}(|M_t|\geq v)\leq 2 e^{-v^2/2NK^2}+N\max_{k\in \{0,\dots, N-1\}}\Prob_{\lfloor N u\rfloor}\big(|M_{t(k+1)/N} - M_{tk/N}|\geq K\big).\]
	Assume that $K>2t$, the probability in the right hand side is less than the probability that the random walk $X$ jumps  at least $K/2$ times in the time interval $[t(k+1)/N, tk/N]$, which is less than $C e^{-K/t}$  since the number of jumps in this interval is distributed as a Poisson random variable of parameter $t$.
	This finally yields
	\[\Prob_{\lfloor N u\rfloor}(|M_t|\geq v)\leq 2 e^{-v^2/2NK^2}+N C e^{-K/t},\] 
	so that, choosing $K=(\log N\log v)^2/t$,  \eqref{estimate1} yields
	\begin{equation}
	\label{estimateXt}
	\Prob_{\lfloor Nu \rfloor}\pa{\sup_{s\leq t}|M_s|\geq N^{3/4}}=\mathcal{O}(N^{-1/8}),
	\end{equation}
which proves \eqref{estimateXt0}.

\medskip

We now get back to \eqref{duality_mean}. For any $u\notin [0,(2p-1)t]$, for any $N$ large enough, according to \eqref{estimateXt0}, $\Prob_{\lfloor Nu \rfloor}(\tau<t)\leq \Prob_{\lfloor Nu \rfloor} (\exists\; s\leq t, X_s=0)=\mathcal{O}(N^{-1/8})$. In particular, recall from \eqref{eq:Defrho} the definition of $\rho$, and recall that we assumed $\rho_0$ to be uniformly $c_0$-Lipschitz, \eqref{duality_mean} yields 

\begin{equation*}
\big|\rho^N_{\lfloor Nu \rfloor}(t)-\rho(t,u)\big|\leq 2N^{-1/4}c_0+\mathcal{O}(N^{-1/8})\underset{N\to\infty}{\longrightarrow}0.
\end{equation*}
We now consider $u\in(0,(2p-1)t)$, for which we can write, using the same bounds, 

\begin{equation*}
\big|\rho^N_{\lfloor Nu \rfloor}(t)-\rho(t,u)\Prob_{\lfloor N u\rfloor}(\tau\geq t)\big|\leq 2N^{-1/4}c_0+\mathcal{O}(N^{-1/8})\underset{N\to\infty}{\longrightarrow}0.
\end{equation*}
To conclude the proof, we only need to show that $\Prob_{\lfloor N u\rfloor}(\tau\geq t)$ converges to $1- \widetilde{\alpha}$  as $N\to\infty.$ To do so, we show that $\Prob_{\lfloor N u\rfloor}(\tau\geq t)=\Prob_{0}(\tau=\infty)+o_N(1)$. We start by writing 
\[\Prob_{\lfloor N u\rfloor}(\tau\geq t)=\Prob_{\lfloor N u\rfloor}(\tau=\infty)+\Prob_{\lfloor N u\rfloor}(t\leq \tau<\infty).\]
By Markov property, the first term in the right hand side is $\Prob_{0}(\tau=\infty)$ because an asymmetric random walk $X$ started from $x>0$ hits the origin a.s.~in finite time. Regarding the second term, we can write using \eqref{estimateXt0} that for some \emph{negative} $u'$,
\[\Prob_{\lfloor N u\rfloor}(t\leq \tau<\infty)\leq \Prob_{\lfloor nu'\rfloor}(\tau<\infty)+\mathcal{O}(N^{-1/8}).\]
Indeed, at time $t$ the random walk has not died, and is w.h.p.~of order $1-\mathcal{O}(N^{-1/8})$ to the right of $\lfloor nu'\rfloor$, where $u'=(u-(2p-1)t)/2<0$. However, the random walk being asymmetric, the probability $\Prob_{\lfloor Nu'\rfloor}(\tau<\infty)$ is $\mathcal{O}(e^{-cN})$ since it visits the origin with probability $\mathcal{O}(e^{-cN})$ for some positive constant $c=c(u',p)$. Finally \eqref{killprob} yields as wanted
\[\Prob_{\lfloor N u\rfloor}(\tau\geq t)=\Prob_{0}(\tau=\infty)+\mathcal{O}(N^{-1/8})=1-\widetilde{\alpha}+\mathcal{O}(N^{-\varepsilon}),\]
where $\varepsilon=|\beta|\wedge 1/8$.

\medskip

All estimates above are uniform away from the boundary points $u=0$, $u=(2p-1)t$, so that the second statement of the lemma is now clear. The last statement of the lemma is straightforward, we do not detail it. 
\end{proof}

\begin{proof}[Proof of Lemma \ref{lem:CorrelationField}]
In the asymmetric case, the proof of Lemma \ref{lem:CorrelationField} is straightforward, we simply sketch it. In the same way that we obtained \eqref{eq:FCdens}, one can write using Dynkin's formula
\begin{equation}
\label{eq:FKcor}
\begin{cases}
\frac{d}{dt} \varphi^N_{x,y}(t) =\mathscr{L}_N^{\dagger,2} \varphi^N_{x,y}(t)& \quad  x,\;y\in \Z, |x-y|>1,\; t>0,\; \\ 
\varphi^N_{\dd,y}(t)=\varphi^N_{x,\dd}(t)=0& \quad  t>0, \;x, \;y\in \Z,\\
\varphi^N_{x,y}(0)=0 & \quad  x,\;y\in \Z,
\end{cases}
\end{equation}
where $\mathscr{L}_N^{\dagger,2}$ is the generator of a two-dimensional asymmetric random walk ${\bf X}=(X,Y)$ on $\Z^2$, where both $X$ and $Y$ are driven by the generator $\mathscr{L}_N^{\dagger}$ defined in \eqref{eq:Ldag}. In other words, under $\mathscr{L}_N^{\dagger,2}$, ${\bf X}$ jumps from $(x,y)$ to either $(x+1,y)$ or $(x, y+1)$ at rate $1-p$, to  
$(x-1,y)$ or $(x, y-1)$ at rate $p$, and jumps from $(x,0)$ or $(0,y)$ to the cemetery state $(x,\dd)$, $(\dd,y)$ at rate $\alpha N^{1+\beta}$. Note that not to burden the proof, we say nothing about the behavior of the random walk ${\bf X}=(X,Y)$ close to the diagonal, \textit{i.e.}~when $|X-Y|\leq 1$. When the random walk gets close to the diagonal, we will simply use the crude bound yielded by attractiveness, namely
\[\sup_{\substack{x,y\in \Z\\ t\geq 0}}|\varphi^N_{x,y}(t)|\leq C(\rho^\star).\]
Once again, by Feynman-Kac's formula, one can write for any $|x-y|>1$, $t\geq 0$,
\begin{equation}
\label{eq:coreq}
\varphi^N_{x,y}(t)=\mathbb{E}_{x,y}\cro{\varphi_{{\bf X}_{\tau'}}(t-\tau')\;{\bb 1}_{\{\tau'<t\wedge \tau\}}}\leq C(\rho^\star)\bb P_{x,y}(\tau'<t),
\end{equation}
where $\bb P_{x,y}$, $\mathbb{E}_{x,y}$ denotes the distribution of the random walk ${\bf X}_t=(X_t,Y_t)$ started from ${\bf X}_0=(x,y)$, and the corresponding expectation. In the identity above, $\tau'$ is the first time the random walk ${\bf X}_t$ hits the set $|x-y|\leq 1$, whereas $\tau=\tau(X)\wedge \tau(Y)$ is the first time either $X$ or $Y$ hits the cemetery state. Equation \eqref{eq:coreq} comes from the fact that we have three boundaries to the system \eqref{eq:FKcor} : if either $X$ or $Y $ hits  the cemetery state, $\varphi^N$ vanishes. If the random walk reaches time $t$, then we query the initial time correlations, which also vanish. The only non-zero contribution therefore comes from getting close to the diagonal. However, given $\varepsilon >0$,  uniformly in $|x-y|\geq \varepsilon N$, $\bb P_{x,y}(\tau'<t)$ vanishes as $N\to\infty$, because $X_t$ and $Y_t$  will be respectively close (up to a fluctuation of order $\mathcal{O}(\sqrt{N})$) to $x+(2p-1)Nt$ and $y+(2p-1)Nt$ and $x$ and $y$ are far apart. This proves \eqref{escor2}.
\end{proof}

\appendix

\section{Equivalence of Definitions \ref{def:ES2} and \ref{def:sol}}
\label{app:equiv}

In Definition \ref{def:ES2} the boundary condition is given by the density $\varrho (t)$ at the origin,  while in Definition \ref{def:sol} the boundary term is given by the flux $f(t)$ into the system. In this section, we show that  the two definitions coincide if $f(t) = (2p-1) \Phi (\varrho (t))$. 
	
\medskip	
	
Indeed, assuming $\rho$ is an entropy solution to \eqref{hE0}, we shall prove it is also an entropy solution to \eqref{hEpos} with $f(t) = (2p-1) \Phi (\varrho (t))$. Then by uniqueness of the weak entropy solutions, these two definitions are the same.  Taking $H \in C_K^{1,1} ((0,\infty) \times (0,\infty))$ in \eqref{hE0Integ1} proves that $\rho$ satisfies \eqref{hEposInteg1}.  It remains to prove $\rho$ also satisfies \eqref{hEposBoundary}. Fix $T> 0$.  We also need to assume furthermore that $\varrho$ has a finite number of discontinuities on the time interval $[0,T]$.  For any interval $[t_i,t_{i+1}]$ and $a > 0$, taking $H(t,u) = \bb{1}_{[t_i,t_{i+1}]} (t) \times \bb{1}_{[-a,a]}(u)$ and $c = \varrho(t_i)$ in \eqref{hE0Integ1}, we have
\begin{equation}\label{equivalence1}
	\begin{aligned}
	\int_0^a &\big\{|\rho(t_i,u) - \varrho(t_i)| - |\rho(t_{i+1},u) - \varrho(t_i)| \big\}du + M \int_{t_i}^{t_{i+1}} |\varrho (t) - \varrho(t_i)| dt\\
	&\geq \int_{t_i}^{t_{i+1}}  (2p-1) \big|\Phi (\rho(t,a) ) - \Phi (\varrho(t_i))\big|dt\\
	&\geq \int_{t_i}^{t_{i+1}}  (2p-1) \big|\Phi (\rho(t,a) ) - \Phi (\varrho(t))\big|dt - \int_{t_i}^{t_{i+1}}  (2p-1) \big|\Phi (\varrho(t) ) - \Phi (\varrho(t_i))\big|dt.
	\end{aligned}
\end{equation}
Since $\Phi$ is smooth,
\[\int_{t_i}^{t_{i+1}}  (2p-1) \big|\Phi (\varrho(t) ) - \Phi (\varrho(t_i))\big|dt\leq C (\Phi) \delta_{\varrho} (t_i,t_{i+1}) (t_{i+1} - t_{i}), \] 
where $\delta_{\varrho} (t_i,t_{i+1})  = \sup_{t \in (t_i,t_{i+1})} |\varrho(t) - \varrho(t_i)|$. Similarly, the second term on the left hand side of \eqref{equivalence1} is bounded by $M \delta_{\varrho} (t_i,t_{i+1}) (t_{i+1} - t_{i})$. Taking a partition $\{t_i\}$ of the time interval $[0,T]$ and summing over the partition in \eqref{equivalence1}, we have
\begin{multline*}
		\sum_{t_i} \int_0^a \big\{|\rho(t_i,u) - \varrho(t_i)| - |\rho(t_{i+1},u) - \varrho(t_i)| \big\}du  + (C (\Phi) + M) T_0 \sup_{i} \delta_{\varrho} (t_i,t_{i+1}) \\
		\geq \int_{0}^{T}  (2p-1) \big|\Phi (\rho(t,a) ) - \Phi (\varrho(t))\big|dt.
\end{multline*}
Since $\inf_{\{t_i\}} \,\sup_{i}\, \delta_{\varrho} (t_i,t_{i+1})  = 0$ and $\rho (t,u)$ is bounded, 
\begin{equation*}
	\begin{aligned}
		Ca &\geq \inf_{\{t_i\}}\,\sum_{t_i} \int_0^a \big\{|\rho(t_i,u) - \varrho(t_i)| - |\rho(t_{i+1},u) - \varrho(t_i)| \big\}du   \\
		&\geq \int_{0}^{T}  (2p-1) \big|\Phi (\rho(t,a) ) - \Phi (\varrho(t))\big|dt.
	\end{aligned}
\end{equation*}
We conclude the proof by letting $a \rightarrow 0$.

\bibliographystyle{plain}
\bibliography{reference.bib}

\begin{thebibliography}{10}

\bibitem{Andjel82}
E.~D. Andjel.
\newblock Invariant measures for the zero range process.
\newblock {\em The Annals of Probability}, 10(3):525--547, 1982.

\bibitem{bahadoran2012hydrodynamics}
C.~Bahadoran.
\newblock Hydrodynamics and hydrostatics for a class of asymmetric particle
  systems with open boundaries.
\newblock {\em Communications in Mathematical Physics}, 310(1):1--24, 2012.

\bibitem{BMN17}
R.~Baldasso, O.~Menezes, A.~Neumann, and R.~Souza.
\newblock Exclusion process with slow boundary.
\newblock {\em Journal of Statistical Physics}, 167(5):1112–1142, Mar 2017.

\bibitem{bardos1979}
C.~Bardos, A.~Y. Leroux, and J.~C. N{\'e}d{\'e}lec.
\newblock First order quasilinear equations with boundary conditions.
\newblock {\em Communications in Partial Differential Equations},
  4(9):1017--1034, 1979.

\bibitem{demasi2021}
A.~De~Masi, S.~Marchesani, S.~Olla, and L.~Xu.
\newblock Quasi-static limit for the asymmetric simple exclusion.
\newblock preprint available on HAL: hal-03177202, March 2021.

\bibitem{diperna1985measure}
R.~J. DiPerna.
\newblock Measure-valued solutions to conservation laws.
\newblock {\em Archive for Rational Mechanics and Analysis}, 88(3):223--270,
  1985.

\bibitem{Erignoux18}
C.~{Erignoux}.
\newblock {Hydrodynamic limit of boundary driven exclusion processes with
  nonreversible boundary dynamics}.
\newblock {\em Journal of Statistical Physics}, 172:1327--1357, 2018.

\bibitem{ELX18}
C.~{Erignoux}, C.~{Landim}, and T.~{Xu}.
\newblock {Stationary states of boundary driven exclusion processes with
  nonreversible boundary dynamics}.
\newblock {\em Journal of Statistical Physics}, 171:599--631, 2018.

\bibitem{FGN15}
T.~Franco, P.~Gon{\c{c}}alves, and A.~Neumann.
\newblock Equilibrium fluctuations for the slow boundary exclusion process.
\newblock In {\em Meeting on Particle Systems and PDE's}, pages 177--197.
  Springer, 2015.

\bibitem{fritz2004}
J.~Fritz.
\newblock {\em Entropy Pairs and Compensated Compactness for Weakly Asymmetric
  Systems}, volume~39 of {\em Advanced Studies in Pure Mathematics}, chapter~6,
  pages 143--171.
\newblock Mathematical Society of Japan, Tokyo, 2004.

\bibitem{fritztoth}
J.~Fritz and B.~Tóth.
\newblock Derivation of the leroux system as the hydrodynamic limit of a
  two-component lattice gas.
\newblock {\em Commun. Math. Phys.}, 249:1--27, 2004.

\bibitem{klscaling}
C.~Kipnis and C.~Landim.
\newblock {\em Scaling limits of interacting particle systems}, volume 320.
\newblock Springer Science \& Business Media, 2013.

\bibitem{kipnis1989hydrodynamics}
C.~Kipnis, S.~Olla, and S.~R.~S. Varadhan.
\newblock Hydrodynamics and large deviation for simple exclusion processes.
\newblock {\em Communications on Pure and Applied Mathematics}, 42(2):115--137,
  1989.

\bibitem{landim1996hydrodynamical}
C.~Landim.
\newblock Hydrodynamical limit for space inhomogeneous one-dimensional totally
  asymmetric zero-range processes.
\newblock {\em The Annals of Probability}, pages 599--638, 1996.

\bibitem{malek1996}
J.~Málek, J.~Nečas, M.~Rokyta, and M.~Růžička.
\newblock {\em Weak and Measure-valued Solutions to Evolutionary PDEs}.
\newblock Chapman and Hall/CRC, 1996.

\bibitem{otto96}
F.~Otto.
\newblock Initial-boundary value problem for a scalar conservation law.
\newblock {\em Comptes rendus de l'Académie des Sciences}, 1 322(8):729--734,
  1996.

\bibitem{rezakhanlou91}
F.~Rezakhanlou.
\newblock Hydrodynamic limit for attractive particle systems on
  $\mathbb{Z}^{d}$.
\newblock {\em Communications in mathematical physics}, 140(3):417--448, 1991.

\bibitem{xu2021hydrodynamic}
L.~Xu.
\newblock Hydrodynamic limit for asymmetric simple exclusion with accelerated
  boundaries.
\newblock {\em arXiv preprint arXiv:2108.09345}, 2021.

\bibitem{yau1991relative}
H.~T. Yau.
\newblock Relative entropy and hydrodynamics of {Ginzburg-Landau} models.
\newblock {\em Letters in Mathematical Physics}, 22(1):63--80, 1991.

\end{thebibliography}

\end{document}